\documentclass[preprint,3p,authoryear,a4paper]{elsarticle}
\usepackage{subfigure, ulem}

\usepackage{amssymb,amsthm, times,amsbsy,amsmath}
\usepackage{delarray,verbatim}
\usepackage[usenames, dvipsnames]{color}

\usepackage{bm}

\usepackage{ifpdf}
\usepackage{color}
\definecolor{webgreen}{rgb}{0,.5,0}
\definecolor{webbrown}{rgb}{.8,0,0}
\definecolor{emphcolor}{rgb}{0.95,0.95,0.95}

\usepackage{hyperref}
\hypersetup{%
	colorlinks=true,
	linkcolor=webbrown,
	filecolor=webbrown,
	citecolor=webgreen,
	breaklinks=true}
\ifpdf \hypersetup{pdftex,
	pdfstartview=FitH, 
	bookmarksopen=true,
	bookmarksnumbered=true
} \else \hypersetup{dvips} \fi

\newcommand{\ud}{{\rm d}}

\newcommand {\B}{\mathcal{B}}

\numberwithin{equation}{section}

\newtheorem{theorem}{Theorem}[section]

\newtheorem{remark}{Remark}[section]
\newtheorem{lemma}{Lemma}[section]

\numberwithin{remark}{section} \numberwithin{proposition}{section}
\numberwithin{corollary}{section}
\newcommand {\R}{\mathbb{R}}

\newcommand {\p}{\mathbb{P}}

\newcommand {\E}{\mathbb{E}}

\newcommand{\diff}{{\rm d}}

\newcommand{\lev}{L\'{e}vy }

\newcommand{\e}{\mathbb{E}}

\begin{document}

\begin{frontmatter}

\title{On optimal joint reflective and refractive dividend \\ strategies in spectrally positive L\'evy models}

\author[UNSW,UdeM]{Benjamin Avanzi\corref{cor}}
\ead{b.avanzi@unsw.edu.au}

\author[Mex]{Jos\'e-Luis P\'erez}
\ead{jluis.garmendia@cimat.mx}

\author[UNSW]{Bernard Wong}
\ead{bernard.wong@unsw.edu.au}

\author[Kan]{Kazutoshi Yamazaki}
\ead{kyamazak@kansai-u.ac.jp}

\cortext[cor]{Corresponding author. }

\address[UNSW]{School of Risk and Actuarial Studies, UNSW Australia Business School, UNSW Sydney NSW 2052, Australia}
\address[UdeM]{D\'epartement de Math\'ematiques et de Statistique, Universit\'e de Montr\'eal, Montr\'eal QC  H3T 1J4, Canada}
\address[Mex]{Department of Probability and Statistics, Centro de Investigaci\'on en Matem\'aticas A.C. Calle Jalisco s/n. C.P. 36240, Guanajuato, Mexico}
\address[Kan]{Department of Mathematics, Faculty of Engineering Science, Kansai University, 3-3-35 Yamate-cho, Suita-shi, Osaka 564-8680, Japan} 

\begin{abstract}
The expected present value of dividends is one of the classical stability criteria in actuarial risk theory. In this context, numerous papers considered threshold (refractive) and barrier (reflective) dividend strategies. These were shown to be optimal in a number of different contexts for bounded and unbounded payout rates, respectively.

In this paper, motivated by the behaviour of some dividend paying stock exchange companies, we determine the optimal dividend strategy when both continuous (refractive) and lump sum (reflective) dividends can be paid at any time, and if they are subject to different transaction rates.

We consider the general family of spectrally positive L\'evy processes. Using scale functions, we obtain explicit formulas for the expected present value of dividends until ruin, with a penalty at ruin. We develop a verification lemma, and show that a two-layer $(a,b)$ strategy is optimal. Such a strategy pays continuous dividends when the surplus exceeds level $a>0$, and all of the excess over $b>a$ as lump sum dividend payments. Results are illustrated.
\end{abstract}


\end{frontmatter}

\section{Introduction}

Actuarial surplus models were originally developed by pioneers \citet*{Lun09} and \citet*{Cra30,Cra55}. The classical surplus model is hence referred to as the `Cram\'er-Lundberg' model. Its `dual' model \citep*[see][]{MaRu04,AvGeSh07} has gained considerable traction in recent years. Besides its intrisic interest, useful links and synergies to the Cram\'er-Lundberg model can be found \citep*[see, e.g.][and also Remark \ref{R_CL}]{AfCaEg13}.

In this paper, we focus on the development of a new type of optimal dividend strategies in the dual model. While the probability of ruin is the traditional measure of stability (see \citet*{AsAl10} for a comprehensive reference and \citet*{Buh70} for a description of the stability problem), an alternative criterion was introduced by \citet*{deF57}. Bruno de Finetti argued the surplus should not implicitly be assumed to go to infinity when studying the probability of ruin---the only way for the probability of ruin not to be certain is for the surplus to grow to infinity. Recognising that surplus leakages were bound to happen, he suggested to study their expected present value instead. This led to a very prolific literature \citep*[see][for comprehensive reviews]{AlTh09,Ava09}.

One has to determine when and how much dividends to pay---the answer to which is referred to as a \textit{dividend strategy}. Traditionally, researchers focused on determining the dividend strategy that maximises the expected present value of dividends (without modifications). Main references are \citet*{Ger69,Loe08} for the Cram\'er-Lundberg model, and \citet*{BaEg08,AvShWo11,BaKyYa12} for the dual model. 
In absence of fixed transaction costs, such optimal strategies are often of the \textit{barrier} (more generally, band) type if dividend payouts are not constrained \citep*[see][for the dual model]{BaKyYa12}, and of \textit{threshold} type if dividend payouts are subject to a maximum payment rate \citep*[see][]{YiWe13}. These two main types lead to \textit{reflection} or \textit{refraction} strategies, respectively; see also \citet*{GeSh06_3}. In presence of fixed transaction costs, impulse strategies appear \citep*[see, e.g.][]{BaKyYa13}.

Often, reflection and refraction strategies are studied separately. However, a combination of both can be observed sometimes in practice, and can be justified from an economic point of view; see also \citet*{GeLo12}, who argue that companies often have distribution strategies that have barrier and threshold components. In essence, the dividend optimisation determines whether surplus dollars are best distributed now, or kept for generating dividends in the future. The existence of a surplus (a buffer) is justified by the existence of frictional costs between capital injections and dividends. In other words, the existence of transaction costs (taxes, financial intermediaries, delays) mean that one should minimise any back and forth movements between companies and their investors. Because of this, if there is not much free money available (if the surplus is low) then it makes sense to minimise distribution. If there is a healthy cushion (arguably the most frequent situation) then a regular flow of dividends can be distributed. If there is a lot of excess money available (due to a particularly profitable year of operation), diminishing marginal profit returns suggest that any excess beyond that comfortable cushion should be distributed to investors. 

The situation we just described is likely to occur when gains follow a stochastic behaviour such as in the dual model. Examples include discovery, commission-based, and mining companies. For instance, in the global mining industry, both BHP Billiton and Rio Tinto (which correspond to two of the largest listed mining companies globally), have declared dividend policies that involve both refractive and reflective components \citep*{BHP16,RioT16}.  Indeed, BHP Billiton states the following in its dividend policy declaration \citep*{BHP16}: ``The Board will assess, every reporting period, the ability to pay amounts \textit{additional to the minimum payment}, in accordance with the capital allocation framework.'' This behaviour is clearly observed in practice. In the United States, `extra dividends' are often paid on top of regular dividends. In this instance, we interpret regular dividends as continuous payments that would be paid under the refraction regime, whereas extra payments are reflective payments. Additionally, `extra' dividends have the advantage of \textit{not} being considered regular (by definition), which presents advantages as investors expect stable, nondecreasing dividend payments \citep*[see, e.g.,][and the references therein]{AvTuWo16c}. Any increase in regular dividends would likely be locked in.

The central question in this paper is how to define the three areas described above in an optimal way. The first area corresponds to no dividend payment, the second to a steady flow of payments---a refraction area, and the third area corresponds to an excess lump sum payment area---a reflection area. The boundaries between the three areas will be denoted $a$ and $b>a$. The reason for the name of ``two-layer $(a,b)$ dividend strategy'' is now transparent, as there are two dividend paying layers: (i) between $a$ and $b$, where dividends are paid continuously and are subject to a maximum payout rate, and (ii) above $b$, where all excess dividends are immediately paid as lump sum dividends to bring back the surplus to level $b$.

As explained earlier, it is very inefficient to distribute money and then raise money again if needed. This is why a money buffer---the surplus---is required. The mathematical way of defining this inefficiency is via transaction costs. Along those lines, it can be arguable that  the cost of a barrier type dividend (modelled as costs based on a singular component of dividend payments) would be higher than the cost of a threshold type dividend (modelled as costs based on the absolutely continuous component). Furthermore, extra dividends have a signalling effect in that they may be interpreted as an indirect commitment to higher returns in future years.

The two-layer $(a,b)$ strategy was first considered by \citet*{Ng10}, who investigated the present value of dividends under a combined threshold and barrier strategy for the compound Poisson dual model, although the issue of optimality was not considered. This is also related to the \textit{multilayer} dividend strategy \citep*[see, e.g.][]{AlHa07}. We consider the much more general case of a spectrally positive \lev process, and investigate the optimality of the two-layer strategy. We show that the global optimality of such a strategy will depend on the relative level of expenses for the different types of dividend payments. The possibility of the optimal strategy involving liquidation, or with persistent refraction but no liquidation, is also highlighted. 

Numerical examples are further used to obtain additional insights. First, we consider a compound Poisson dual model with diffusion and phase-type jumps, that admits an analytical form of scale function; see also \ref{S_PTLP}.  In particular, our problem is a mixture of the cases with only reflective and refractive dividend payments as in \citet*{BaKyYa13} and \citet*{YiWe13} in the dual model.  We confirm through a series of numerical examples that the solutions approach to those in \citet*{BaKyYa13} and \citet*{YiWe13} as the parameters describing the problem approach to some limits. Then with exponential jumps and no diffusion \citep*[such as in][for instance]{AvGeSh07}, we study the impact of the volatility of stochastic gains on the optimal dividend strategy.

The paper is arranged as follows. In section \ref{S_Model} the surplus model is formally introduced.  A useful link between our model and a version of the optimal dividend problem  with additional  capital injection of threshold-type (with ruin) is also made.  Section \ref{S_EPV}
introduces the two-layer $(a,b)$ strategy formally and develops explicit formulas for the expected present value of dividends until ruin and payoff at ruin with the help of scale functions.  Existence and uniqueness of optimal parameters $(a^*,b^*)$, which satisfy certain smoothness conditions, are established in Section \ref{section_a_b}.  A verification lemma is developed in Section \ref{S_opti}, which is subsequently used to show that the two-layer $(a^*,b^*)$ strategy is indeed optimal. Numerical illustrations are discussed in Section \ref{numerical_section}. Section \ref{S_conclusion} concludes.

\section{Model} \label{S_Model}

\subsection{Surplus model}

We consider the surplus $Y=(Y_t; t\geq 0)$ which is a L\'evy process defined on a  probability space $(\Omega, \mathcal{F}, \p)$.  For $x\in \R$, we denote by $\p_x$ the law of $Y$ when it starts at $x$ and write for convenience  $\p$ in place of $\p_0$. Accordingly, we write $\e_x$ and $\e$ for the associated expectation operators. We shall assume throughout that $Y$ is spectrally positive, meaning here that it has no negative jumps and that it is not a subordinator. This allows us to define the Laplace exponent $\psi_Y(\theta):[0,\infty) \to \R$, i.e.
	\[
	\e\big[e^{-\theta Y_t}\big]=:e^{\psi_Y(\theta)t}, \qquad t, \theta\ge 0,
	\]
	given by the L\'evy-Khintchine formula
	\begin{equation}
		\psi_Y(\theta):=\gamma_Y\theta+\frac{\sigma^2}{2}\theta^2+\int_{(0,\infty)}\big(e^{-\theta x}-1+\theta x\mathbf{1}_{\{x<1\}}\big)\Pi(\ud x), \quad \theta \geq 0,\label{Laplace_Y}
	\end{equation}
	where $\gamma_Y \in \R$, $\sigma\ge 0$, and $\Pi$ is a measure on $(0,\infty)$ called the L\'evy measure of $Y$ that satisfies
	\[
	\int_{(0,\infty)}(1\land x^2)\Pi(\ud x)<\infty.
	\]

It is well-known that $Y$ has paths of bounded variation if and only if $\sigma=0$ and $\int_{(0,1)} x\Pi(\mathrm{d}x) < \infty$; in this case, $Y$ can be written as
\begin{equation}
Y_t=-c_Yt+S_t, \,\,\qquad t\geq 0,\notag
\end{equation}
where 
\begin{align}
c_Y:=\gamma_Y+\int_{(0,1)} x\Pi(\mathrm{d}x) \label{def_drift_finite_var}
\end{align}
 and where $(S_t; t\geq0)$ is a driftless subordinator. Note that  necessarily $c_Y>0$, since we have ruled out the case that $Y$ has monotone paths. Its Laplace exponent is given by
\begin{equation*}
\psi_Y(\theta) = c_Y \theta+\int_{(0,\infty)}\big( e^{-\theta x}-1\big)\Pi(\ud x), \quad \theta \geq 0.
\end{equation*}

\subsection{Introduction of dividends}
A dividend strategy  is a pair $\pi := \left( A_t^{\pi}, S_t^{\pi}; t \geq 0 \right)$ of nondecreasing, right-continuous, and adapted processes (with respect to the filtration generated by $Y$) such that $A_{0-}^\pi = S_{0-}^\pi = 0$.  In addition,  with $\delta > 0$ fixed, we require that $A^\pi$ is absolutely continuous with respect to the Lebesgue measure of the form $A_t^\pi = \int_0^t a^\pi_s \diff s$, $t \geq 0$, with $a^\pi$ restricted to take values in $[0,\delta]$ uniformly in time. The controlled risk process becomes
\begin{align*}
U_t^\pi &:= Y_t - A_t^\pi - S_t^\pi, \quad t \geq 0.
\end{align*}
Here and throughout the paper, let $\Delta \zeta_t = \zeta_t - \zeta_{t-}$ for any process $\zeta$.
Let $\mathcal{A}$ be the set of all admissible strategies that satisfy the above conditions and\begin{align*}
\Delta S_t^\pi \leq U_{t-}^\pi + \Delta Y_t, \quad t \geq 0.
\end{align*}
Proportional costs are incurred for $A^\pi$ and $S^\pi$, and are denoted $\varepsilon_A$ and $\varepsilon_S$, respectively.  As explained in the introduction, we assume that the former is less than the latter:
\begin{align*}
\varepsilon_A < \varepsilon_S,
\end{align*}
with net dividend rates
\begin{align*}
\beta_A := 1-\varepsilon_A > 1- \varepsilon_S =: \beta_S,
\end{align*}
respectively. We additionally assume $\varepsilon_S < 1$ so that $\beta_A > \beta_S > 0$.  Note that $\varepsilon_A, \varepsilon_S$ could technically be negative.

We want to maximize over all admissible strategies $\pi$ the total expected dividends  until ruin less transaction costs and terminal payoff  $\tilde{\rho}\in \R$ (more precisely a penalty if negative): with $q > 0$,
\begin{align}\label{E_Obj1}
V_{\pi} (x) := \mathbb{E}_x \left( \beta_A \int_0^{\sigma^\pi} e^{-q t} a_t^\pi  \diff t + \beta_S \int_{[0, \sigma^\pi]} e^{-q t} \diff S_t^{\pi} + \tilde{\rho} e^{-q \sigma^\pi}\right), \quad x \geq 0,
\end{align}
where
\begin{align*} 
\sigma^\pi &:= \inf \{ t > 0: U_t^\pi < 0 \},
\end{align*}
is the time to ruin. Here and throughout the paper, let $\inf \varnothing = \infty$.

For the rest of the paper, for notational convenience we will consider the equivalent problem of maximizing
\begin{align}
v_\pi (x):= \frac {V_\pi(x)} {\beta_A}=\mathbb{E}_x \left( \int_0^{\sigma^\pi} e^{-q t} a_t^\pi  \diff t + \beta \int_{[0, \sigma^\pi]} e^{-q t} \diff S_t^{\pi} + \rho e^{-q \sigma^\pi}\right), \quad x \geq 0, \label{v_pi}
\end{align}
where $\beta:= \beta_S / \beta_A \in (0, 1)$ and $\rho := \tilde{\rho} / \beta_A \in \R$. 
Our objective is to compute the \textit{value function}
\begin{align*}
v(x) = \sup_{\pi \in \mathcal{A}} v_\pi (x), \quad x \geq 0,
\end{align*}
and obtain the optimal strategy that attains it, if such a strategy exists.
It is clear that $\beta_A v(x) = \sup_{\pi \in \mathcal{A}} V_\pi (x)$, which are both maximized by the same strategy.

For the rest of the paper, we assume that the expected profit per unit of time (drift of the surplus before dividends)
	\begin{align}
	\E [Y_1] = - \psi_Y'(0+) < \infty,
	\end{align}
	so that the considered problem will have nontrivial solutions.

\begin{remark}
There are alternative ways of interpreting the two-layer $(a,b)$ strategy. Assume that the deterministic drift in the middle layer  is the actual expense rate. The top layer corresponds then to a `traditional' barrier strategy. The main difference is in the bottom layer, where the drift is effectively pushed up by continuous capital injections. As shown below, maximisation of dividends net of capital injections in this context is a problem which is equivalent to that considered in this paper.

The underlying surplus process is again a spectrally positive \lev process $\hat{Y}$.   The strategy is given by a pair $\vartheta = (L^\vartheta, S^\vartheta)$ where $L^{\vartheta}$ models the capital injection while $S^\vartheta$ models the (usual) dividend.  We require that $L^\vartheta$ (as in $A^\pi$ above) is absolutely continuous with respect to the Lebesgue measure of the form $L_t^\vartheta = \int_0^t l^\vartheta_s \diff s$, $t \geq 0$, with $l^\vartheta$ restricted to take values in $[0,\delta]$ uniformly in time. The controlled surplus process is then $U_t^\vartheta = \hat{Y}_t - S_t^\vartheta + \hat{L}_t^{\vartheta}$, $t \geq 0$.

 The problem is  to maximize, over all admissible strategies $\vartheta$, the total dividend until ruin less capital injection with additional terminal payoff: for $\beta_A > \beta_S$ and $\hat{\rho} \in \R$,
\begin{align*}
\hat{V}_{\vartheta} (x) := \mathbb{E}_x \left(  \beta_S \int_{[0, \sigma^\vartheta]} e^{-q t} \diff S_t^{\vartheta} - \beta_A \int_0^{\sigma^\vartheta} e^{-q t} l_t^\vartheta \diff t + \hat{\rho} e^{-q \sigma^\vartheta} \right), \quad x \geq 0,
\end{align*}
where
\begin{align*}
\sigma^\vartheta &:= \inf \{ t > 0: U_t^\vartheta < 0 \}.
\end{align*}

This can be easily transformed to the problem of maximizing \eqref{E_Obj1} above.  To see this,  define
\begin{align*}
 Y_t := \hat{Y}_t + \delta t \quad \textrm{and} \quad A_t^\vartheta := \delta t - L_t^\vartheta, \quad t \geq 0.
\end{align*}
Then $U_t^\vartheta = (\hat{Y}_t + \delta t) - S_t^\vartheta + (\hat{L}_t^\vartheta - \delta t) = Y_t - S_t^\vartheta - A_t^\vartheta$ and for all $x \geq 0$
\begin{align*}
\hat{V}_{\vartheta} (x) &= \mathbb{E}_x \left(  \beta_S \int_{[0, \sigma^\vartheta]} e^{-q t} \diff S_t^{\vartheta} + \beta_A \int_0^{\sigma^\vartheta} e^{-q t}  \diff A_t^\vartheta  - \beta_A \delta \int_0^{\sigma^\vartheta} e^{-q t}  \diff t + \hat{\rho} e^{-q \sigma^\vartheta} \right) \\
&= \mathbb{E}_x \left(  \beta_S \int_{[0, \sigma^\vartheta]} e^{-q t} \diff S_t^{\vartheta} + \beta_A \int_0^{\sigma^\vartheta} e^{-q t}  \diff A_t^\vartheta  + \tilde{\rho} e^{-q \sigma^\vartheta}    \right) - \frac {\beta_A \delta} q,
\end{align*}
with $\tilde{\rho} = \hat{\rho}+ \beta_A \delta / q$.
In other words, the problem reduces to maximizing \eqref{E_Obj1}.   Here, for $Y$ to be a spectrally positive \lev process (so that it is not a subordinator), it is required, when $\sigma = 0$, the drift parameter $\hat{c}_Y$ of $\hat{Y}$ must be larger than $\delta$.

\end{remark}

\section{The two-layer $(a,b)$ strategy}
\label{S_EPV}

The objective of this paper is to show the optimality of what we call the \textit{two-layer $(a,b)$ strategy} for a suitable choice of $(a,b)$.  In this section, we shall first define the strategy and compute the expected present value \eqref{v_pi} under this strategy.  Toward this end, we give a brief review of the \textit{refracted-reflected \lev process} of \citet*{PeYa15} and also the \textit{scale function}.

\subsection{Definition of the two-layer $(a,b)$ strategy}\label{S_abstrat} 

We conjecture that the type of strategy that will maximise \eqref{E_Obj1} is such that
\begin{enumerate}
\item (no-dividend area) dividends are not paid when the surplus is less than some level $a \geq0$;
\item  (refractive area) dividends are paid continuously at maximum rate $\delta$ when the surplus is  
more than $a$, but less than $b>a$;
\item (reflective area) dividends are paid as a lump sum such that the process is reflected at $b$ if the process exceeds $b$; in other words, any excess is immediately paid as a lump sum.
\end{enumerate}
Such a strategy will be called a two-layer $(a,b)$ strategy and is denoted by $\pi_{a,b}$ throughout this paper. This is illustrated on the left hand side of Figure \ref{F_samplepaths}. 

\begin{remark}\label{R_CL}
It is remarkable that maximisation of \eqref{E_Obj1} as discussed in this paper is a problem that is intimately related to the maximisation of dividends under a threshold strategy in a spectrally negative model with (forced) capital injections. This is because one controlled sample path is a reflection of the other; see Figure \ref{F_samplepaths}. Thanks to this nice property, some quantities can be drawn directly from \citet*{PeYa15} after appropriate change of variables. Note that neither optimal strategy parameters nor optimality of the strategy were considered in \citet*{PeYa15}. In any case, optimisation would be different, because ruin occurs in this paper's framework, whereas it never occurs in the framework of \citet*{PeYa15}.
\end{remark}
\begin{figure}[htb]
\begin{center}
\subfigure{\includegraphics[height=6cm]{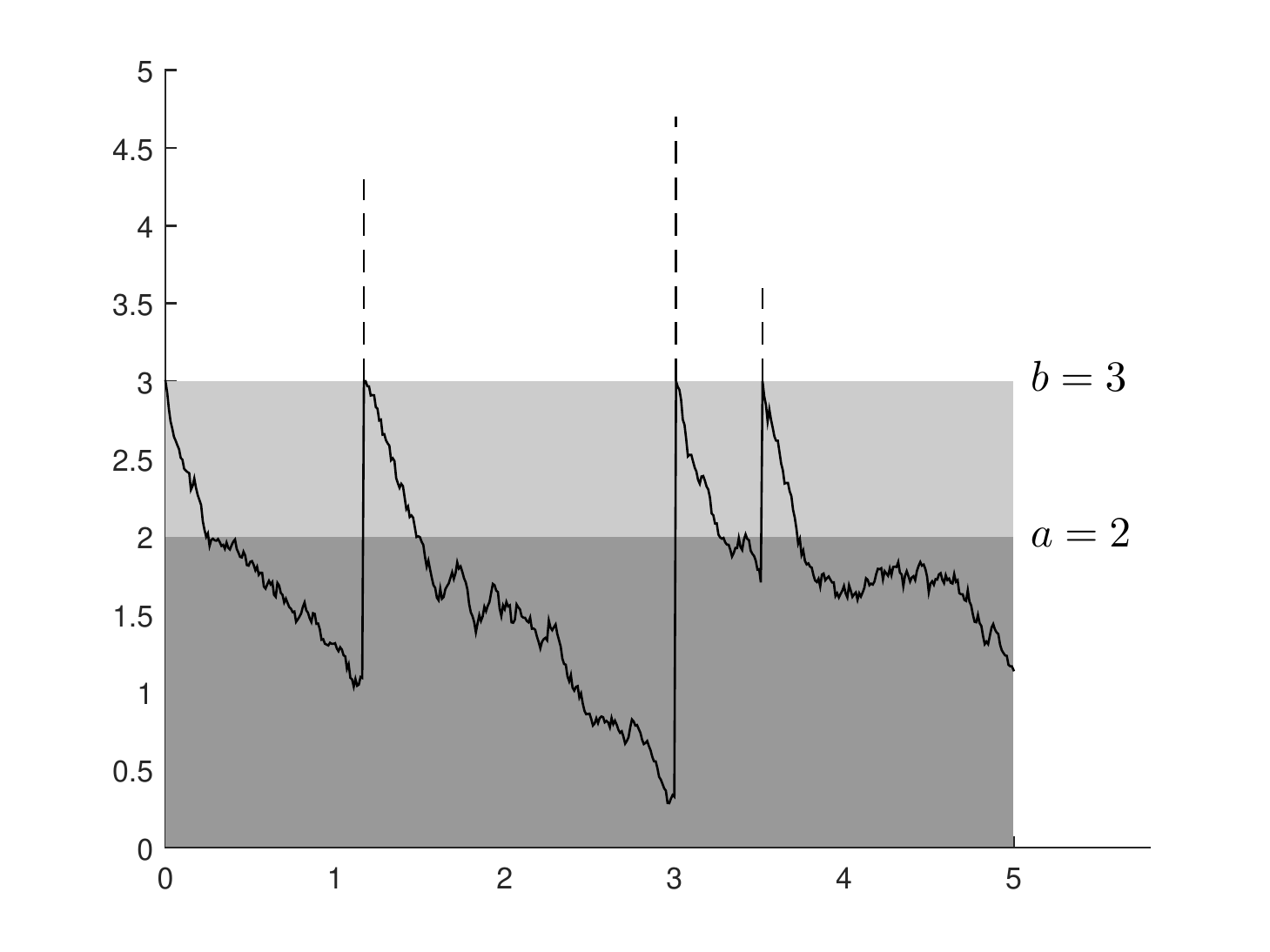}}\hfill
\subfigure{\includegraphics[height=6cm]{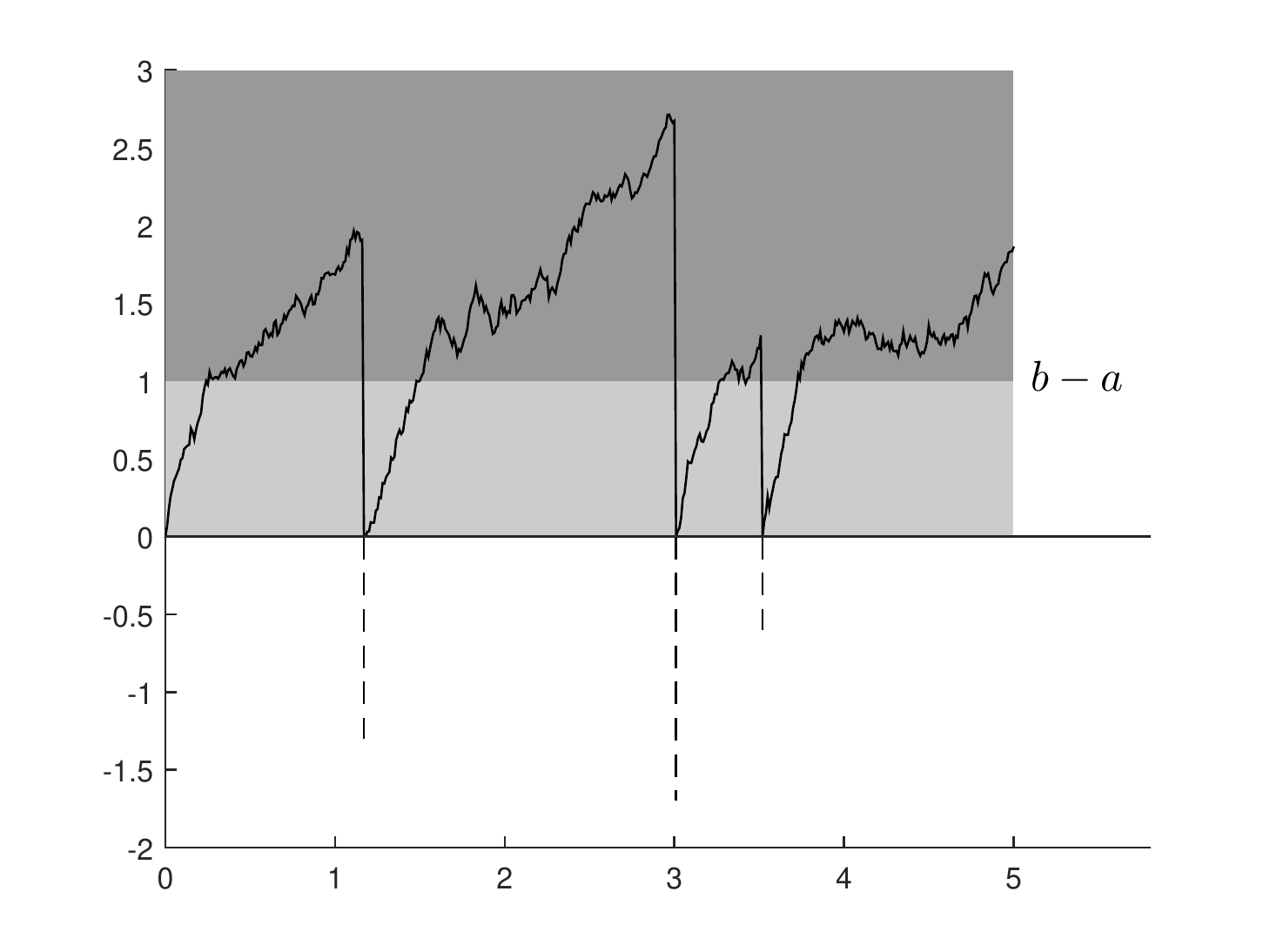}}
\end{center}
\caption{A sample path of $U^{a,b}$ under a two-layer $(a,b)$ strategy (left), and its corresponding spectrally negative refracted-reflected \lev process $\tilde{U}^{0,b-a} \equiv b - U^{a,b}$ (right) as in \eqref{relation_U_U_tilde}.}
\label{F_samplepaths}
\end{figure}

\subsection{Surplus after dividends}

The controlled surplus process, after the distribution of dividends according to a two-layer $(a,b)$ strategy, is denoted by
	\begin{align}
	U^{a,b}_t=Y_t-A^{a,b}_t-S_t^{a,b}, \quad t \geq 0,  \label{U_a_b}
	\end{align}
	where $A_t^{a,b}$ and $S_t^{a,b}$ are the (cumulative) amounts of dividends collected until $t \geq 0$ in the refractive and reflective areas, respectively.  In particular, we can write
		\begin{equation*}
		A^{a,b}_t :=\delta\int_0^t1_{\{U^{a,b}_s >a\}} \diff s, \quad t \geq 0.
	\end{equation*}
	These processes are well-defined by the following arguments.
	
	For fixed $\delta> 0$ and $a \geq 0$, the \textit{refracted spectrally positive \lev process $H$}  is defined as the unique strong solution to the stochastic differential equation (SDE)
\begin{equation}
 H_t= Y_t- \delta \int_0^t 1_{\{H_s>a\}}\diff s,\qquad\text{$t\geq0$;}\label{refracted_levy}
 \end{equation}
 see \citet*{KyLo10} for the spectrally negative \lev case and  \citet*{PeYa15b} for the discussion on the existence/uniqueness for the spectrally positive case.
 Informally speaking, a linear drift at rate $\delta$ is subtracted from the increments of  the underlying \lev process $Y$ whenever it exceeds $a$.  As a result, the process moves like $Y$ below $a$, and above $a$ it moves like a drift-changed process
 \begin{equation}
X_t:=Y_t-\delta t,\qquad\textbf{$t\geq0$}.\label{def_X}
\end{equation}

We also define, for the process $X$, the \textit{\lev process reflected at the upper boundary $b > a$} given by
\begin{align}
V_t &:= X_t - \sup_{0 \leq s \leq t} (X_s - b) \vee 0.  \label{def_V}
\end{align}
 The supremum term pushes the process downward whenever it attempts to up-cross the level $b$; as a result the process only takes values on $(-\infty, b]$. 
	
The process $U^{a,b}$ in \eqref{U_a_b} is a concatenation of the refracted process $H$ and the reflected process $V$.  It is well-defined because it is exactly the dual of the process recently considered in \citet*{PeYa15}; see also Remark \ref{R_CL}. 
To be more precise, we can write \begin{align}
	 U^{a,b}_t = b-\tilde{U}^{0,b-a}_t, \quad t \geq 0, \label{relation_U_U_tilde}
	 \end{align}
where $\tilde{U}^{0,b-a}$ corresponds to a refracted-reflected L\'evy process of \citet*{PeYa15}, driven by the spectrally negative L\'evy process $-X$ that is reflected from below at $0$ and refracted  at the threshold $b-a$.   It admits a decomposition
		\[
	\tilde{U}^{0,b-a}_t= - X_t-\tilde{A}^{0,b-a}_t+\tilde{S}^{0,b-a}_t, \quad t \geq 0,
	\]
		where we can write
		\begin{equation*}
			\tilde{A}^{0,b-a}_t =\delta\int_0^t1_{\{\tilde{U}^{0,b-a}_s>b-a\}} \diff s, \quad t \geq 0,
		\end{equation*}
and where $\tilde{S}^{0,b-a}$ pushes the process upward so that it does not go below zero. In other words, below $b-a$, it moves like a reflected process driven by $-X$ while, above $b-a$, it moves like a spectrally negative \lev process $-Y$. Figure \ref{F_samplepaths} shows the pathwise relation \eqref{relation_U_U_tilde} between $U^{a,b}$ and $\tilde{U}^{0,b-a}$.

For the rest of this paper, let, for all $0 \leq a < b$ and $x \geq 0$, the expected present value of dividends and additional terminal payoff \eqref{v_pi} under  the two-layer $(a,b)$ strategy be denoted
\begin{align}
		v_{a,b}(x)&:= \E_x\left(\int_0^{\sigma_{a,b}}e^{-qt}\diff A_t^{a,b}+\beta\int_{[0,\sigma_{a,b}]}e^{-qt}\diff S^{a,b}_t + \rho e^{-q \sigma_{a,b}}  \right), \label{v_a_b}
\end{align}
with its ruin time
\begin{align*}
\sigma_{a,b} := \sigma^{\pi_{a,b}} = \inf \{ t > 0: U_t^{a,b} < 0 \}.
\end{align*}
By the relationship \eqref{relation_U_U_tilde} and the results in  \citet*{PeYa15}, this can be computed explicitly.  To this end, we first review scale functions.

\subsection{Scale functions} 
Scale functions are the primary tool to derive \eqref{v_a_b}. We review some of the most relevant results in this section, so that we can easily refer to them later. For a comprehensive discussion on the scale function and its applications, we refer the reader to \citet*{Kyp06} and \citet*{KuKyRi13}.

Fix $q > 0$. 
We use $\mathbb{W}^{(q)}$ and $W^{(q)}$, respectively, for the scale functions of the spectrally negative \lev processes $-Y$ and  $-X$ (see \eqref{def_X} for the latter).  These are the mappings from $\R$ to $[0, \infty)$ that take value zero on the negative half-line, while on the positive half-line they are continuous and strictly increasing functions that are defined by their Laplace transforms:
\begin{align} \label{scale_function_laplace}
\begin{split}
\int_0^\infty  \mathrm{e}^{-\theta x} \mathbb{W}^{(q)}(x) \diff x &= \frac 1 {\psi_Y(\theta) -q}, \quad \theta > \varphi(q), \\
\int_0^\infty  \mathrm{e}^{-\theta x} W^{(q)}(x) \diff x &= \frac 1 {\psi_X(\theta)-q}, \quad \theta > \Phi(q), 
\end{split}
\end{align}
where $\psi_Y$ is defined in \eqref{Laplace_Y}, $\psi_X(\theta) := \psi_Y(\theta) + \delta \theta$, $\theta \geq 0$, is the Laplace exponent for $X$, and
\begin{align}
\begin{split}
\varphi(q) := \sup \{ \lambda \geq 0: \psi_Y(\lambda) = q\} \quad \textrm{and} \quad  \Phi(q) := \sup \{ \lambda \geq 0: \psi_X(\lambda) = q\}. \notag
\end{split}
\end{align}
By the strict  convexity of $\psi_Y$, we derive the inequality $\varphi(q) > \Phi(q) > 0$. We also define $c_X$ and $\gamma_X$ for $\psi_X$ analogously to $c_Y$ and $\gamma_Y$.

We also define, for $x \in \R$, 
\begin{align*}
\overline{\mathbb{W}}^{(q)}(x) &:=  \int_0^x \mathbb{W}^{(q)}(y) \diff y, \\
\mathbb{Z}^{(q)}(x) &:= 1 + q \overline{\mathbb{W}}^{(q)}(x),  \\
\overline{\mathbb{Z}}^{(q)}(x) &:= \int_0^x \mathbb{Z}^{(q)} (z) \diff z = x + q \int_0^x \int_0^z \mathbb{W}^{(q)} (w) \diff w \diff z.
\end{align*}
Noting that $\mathbb{W}^{(q)}(x) = 0$ for $-\infty < x < 0$, we have
\begin{align}
\overline{\mathbb{W}}^{(q)}(x) = 0, \quad \mathbb{Z}^{(q)}(x) = 1,  \quad \textrm{and} \quad \overline{\mathbb{Z}}^{(q)}(x) = x, \quad x \leq 0.  \label{z_below_zero}
\end{align}
In addition, we define $\overline{W}^{(q)}$, $Z^{(q)}$ and $\overline{Z}^{(q)}$ analogously for $-X$.

Regarding their asymptotic values as $x \downarrow 0$ we have, as in Lemmas 3.1 and 3.2 of \citet*{KuKyRi13}, 
\begin{align}\label{eq:Wqp0}
\begin{split}
 \mathbb{W}^{(q)} (0) &= \left\{ \begin{array}{ll} 0 & \textrm{if $Y$ is of unbounded
variation,} \\ c_Y^{-1} & \textrm{if $Y$ is of bounded variation,}
\end{array} \right. \\
  W^{(q)} (0) &= \left\{ \begin{array}{ll} 0 & \textrm{if $X$ is of unbounded
variation,} \\ c_X^{-1} & \textrm{if $X$ is of bounded variation,}
\end{array} \right.  
\end{split}
\end{align}
and 
\begin{align*}
\mathbb{W}^{(q) \prime} (0+) &:= \lim_{x \downarrow 0}\mathbb{W}^{(q)\prime} (x+) =
\left\{ \begin{array}{ll}  \frac 2 {\sigma^2} & \textrm{if }\sigma > 0, \\
\infty & \textrm{if }\sigma = 0 \; \textrm{and} \; \Pi(0,\infty) = \infty, \\
\frac {q + \Pi(0, \infty)} {c_Y^2} &  \textrm{if }\sigma = 0 \; \textrm{and} \; \Pi(0,\infty) < \infty,
\end{array} \right. \\
W^{(q)\prime} (0+) &:= \lim_{x \downarrow 0}W^{(q) \prime} (x+) =
\left\{ \begin{array}{ll}  \frac 2 {\sigma^2} & \textrm{if }\sigma > 0, \\
\infty & \textrm{if }\sigma = 0 \; \textrm{and} \; \Pi(0, \infty) = \infty, \\
\frac {q + \Pi(0, \infty)} {c_X^2} &  \textrm{if }\sigma = 0 \; \textrm{and} \; \Pi(0, \infty) < \infty.
\end{array} \right. 
\end{align*}

\begin{remark} \label{remark_smoothness}
Scale functions are reasonably smooth. Indeed, if $Y$ is of unbounded variation or the \lev measure is atomless, it is known that $\mathbb{W}^{(q)}$ and $W^{(q)}$ are $C^1(\R \backslash \{0\})$.  For more comprehensive results on smoothness of scale functions, see \citet*{ChKySa11}.
\end{remark}

\subsection{Expected present value of dividends with terminal payoff/penalty}

We shall now obtain the expected present value of dividends and terminal payoff/penalty \eqref{v_a_b} using the scale function. Define, for all $z, c \in \R$, 
		\begin{align*}
				r^{(q)}_{c}(z)&:= Z^{(q)}(z)+q\delta\int_{c}^{z}\mathbb{W}^{(q)}(z-y)W^{(q)}(y)\ud y, \\
		\tilde{r}^{(q)}_{c}(z)
		&:= R^{(q)} (z) +\delta\int_{c}^z\mathbb{W}^{(q)}(z-y)Z^{(q)}(y)\ud y, 
	\end{align*}
with
\begin{align*}
R^{(q)} (z) := \overline{Z}^{(q)}(z) + \frac {\psi_X'(0+)} q, \quad z \in \R.
\end{align*}
The following is a direct consequence of the results obtained in \citet*{PeYa15}. 
\begin{lemma}  \label{vf_sp} For all $0 \leq a < b$ and $x \geq 0$, we have

\begin{align}
		v_{a,b}(x)
	&=	 -  \frac {\Gamma(a,b)} q \frac{r^{(q)}_{b-a}(b-x)}{r^{(q)}_{b-a}(b)} + \frac {\delta} q \mathbb{Z}^{(q)}(a-x) -\beta  \tilde{r}^{(q)}_{b-a}(b-x),
		 \label{v_a_b_2}
	\end{align}
	where we define, for $0 \leq a \leq b$,
\begin{align} \label{big_gamma}
\begin{split}
\Gamma(a,b) &:=  \delta \mathbb{Z}^{(q)}(a)-q \rho  -q\beta\tilde{r}^{(q)}_{b-a}(b).
\end{split}
\end{align}
In particular for $x \geq b$, $v_{a,b}(x)= \beta(x-b)+v_{a,b}(b)$.
\end{lemma}
	
\begin{proof}

Let $\tilde{\E}_{b-x}$ be the expectation under which $\tilde{U}_{0-}^{0,b-a} = - X_0 = - Y_0 = b-x$ and 
	\[
	\tilde{\tau}_b^+:=\inf\{t>0:\tilde{U}^{0,b-a}_t>b\};
	\]
see also Figure \ref{F_samplepaths}.
Using the relation \eqref{relation_U_U_tilde} and Proposition 5.1 in \citet*{PeYa15}, we obtain that
	\begin{align}\label{capital cost}
		\begin{split}
			\E_x&\left(\int_{[0,\sigma_{a,b}]}e^{-qt}\diff S^{a,b}_t\right) =\tilde{\E}_{b-x}\left(\int_{[0,\tilde{\tau}_b^+]}e^{-qt}\diff \tilde{S}^{0,b-a}_t\right)=\tilde{r}^{(q)}_{b-a}(b) \frac{r^{(q)}_{b-a}(b-x)} {r^{(q)}_{b-a}(b)}- \tilde{r}^{(q)}_{b-a}(b-x).
		\end{split}
	\end{align}
Also, by Corollary 4.2 of \citet*{PeYa15},
	\begin{align}
\mathbb{E}_x\left(e^{-q \sigma_{a,b}} \right)=\tilde{\mathbb{E}}_{b-x}\left(e^{-q \tilde{\tau}_b^+} \right)=\frac{r^{(q)}_{b-a}(b-x)}{r^{(q)}_{b-a}(b)}.\notag
\end{align}
	In a similar way, by using the proof of Corollaries 4.2 and 4.3 in \citet*{PeYa15},
		\begin{align*} 
		\E_x\left(\int_{[0,\sigma_{a,b}]}e^{-qt}\diff A^{a,b}_t\right)&=\delta \tilde{\mathbb{E}}_{b-x}\left(\int_{[0,\tilde{\tau}_b^+]}e^{-qt}\diff t\right)-\tilde{\mathbb{E}}_{b-x}\left(\int_{[0,\tilde{\tau}_b^+]}e^{-qt}\diff\tilde{A}^{0,b-a}_t\right) \\
		&=\frac{\delta}{q}- \frac{\delta}{q} \mathbb{Z}^{(q)}(a) \frac{r^{(q)}_{b-a}(b-x)}{r^{(q)}_{b-a}(b)}+ \delta \overline{\mathbb{W}}^{(q)}(a-x). 
	\end{align*}
Putting the pieces together we obtain the result.
\end{proof}

Note that 
for all $0 \leq a < b$ and $x \in \R$,
	\begin{align*}
			r^{(q)}_{b-a}(b-x)
		&= Z^{(q)}(b-x)+q\delta\int_{0}^{a-x}\mathbb{W}^{(q)}(y)W^{(q)}(b-x-y)\ud y, \\
		\tilde{r}^{(q)}_{b-a}(b-x)
				&= R^{(q)} (b-x) +\delta\int_0^{a-x}\mathbb{W}^{(q)}(y)Z^{(q)}(b-x-y)\ud y, 
	\end{align*}
where in particular
	\begin{align*}
			r^{(q)}_{b-a}(b)
		&= Z^{(q)}(b)+q\delta\int_{0}^{a}\mathbb{W}^{(q)}(y)W^{(q)}(b-y)\ud y, \\
		\tilde{r}^{(q)}_{b-a}(b)
				&= R^{(q)} (b) +\delta\int_0^{a}\mathbb{W}^{(q)}(y)Z^{(q)}(b-y)\ud y. 
	\end{align*}
These results are useful in the later sections.

\section{Optimal layer levels $(a^*, b^*)$} \label{section_a_b}
	
	Based on our conjecture that a two-layer $(a,b)$ strategy is optimal, the first step is to identify the candidates $(a^*, b^*)$ so that the degree of smoothness of $v_{a^*, b^*}$ at these points increases by one.  As we will see in the verification step, this will guarantee that the slope of $v_{a^*, b^*}$ becomes $1$ at $a^*$ and $\beta$ at $b^*$.

\subsection{Smooth fit conditions}

	We will find conditions on the thresholds $a$ and $b$ to ensure that the function $v_{a,b}$ as in \eqref{v_a_b_2} is smooth enough. This translates into the requirement that it is continuously differentiable in the bounded variation case, and that it is twice continuously differentiable in the unbounded variation case.
	\par We begin by computing the first derivative of  \eqref{v_a_b_2}. By the continuity of the scale function on $(0, \infty)$, the function $v_{a,b}$ is differentiable on $(0, \infty) \backslash \{a^*, b^*\}$ with 
	\begin{align}
v_{a,b}'(x)&= \frac {\Gamma(a,b)} q \frac {r^{(q)\prime}_{b-a}(b-x)} {r^{(q)}_{b-a}(b)} - \delta \mathbb{W}^{(q)}(a-x) + \beta \tilde{r}^{(q)\prime}_{b-a}(b-x). \label{v_a_b_derivative}
\end{align}
Here,  for $x \in (0, \infty) \backslash \{a^*, b^*\}$,
\begin{align*}
r^{(q) \prime}_{b-a}(b-x) &= q \Big[ W^{(q)}(b-x)+  \delta \Big( \mathbb{W}^{(q)}(0)W^{(q)}(b-x)+\int_x^a\mathbb{W}^{(q)\prime}(u-x)W^{(q)}(b-u) \diff u \Big) 1_{\{x < a\}} \Big], \\
\tilde{r}^{(q) \prime}_{b-a}(b-x) &= Z^{(q)}(b-x)+ \delta \Big( \mathbb{W}^{(q)}(0)Z^{(q)}(b-x)+\int_x^a\mathbb{W}^{(q)\prime}(u-x)Z^{(q)}(b-u)\diff u \Big) 1_{\{x < a\}},
\end{align*}
and, in particular, 
\begin{align}
r^{(q) \prime}_{b-a}(0+) &= qW^{(q)}(0), \quad \tilde{r}_{b-a}^{(q) \prime}(0+) = 1,  \label{r_prime_zero}
\end{align}
and, for all $b > a \geq 0$, 
\begin{align} \label{r_prime_b_a}
\begin{split}
r^{(q) \prime}_{b-a}((b-a)+)  - r^{(q) \prime}_{b-a}((b-a)-) &= q \delta W^{(q)}(b-a) \mathbb{W}^{(q)}(0), \\ 
\tilde{r}^{(q) \prime}_{b-a}((b-a)+) - \tilde{r}^{(q) \prime}_{b-a}((b-a)-) &= \delta\mathbb{W}^{(q)}(0)Z^{(q)}(b-a).
\end{split}
\end{align}
Furthermore, \textit{for the case of unbounded variation} where $W^{(q)}$ and $\mathbb{W}^{(q)}$ are both differentiable on $\R \backslash \{0\}$ by Remark \ref{remark_smoothness},  the function $v_{a,b}$ is twice-differentiable on $(0, \infty) \backslash \{a^*, b^*\}$ with
	\begin{align}
v_{a,b}''(x)&= -\frac {\Gamma(a,b)} q \frac {r^{(q)\prime \prime}_{b-a}(b-x)} {r^{(q)}_{b-a}(b)} + \delta \mathbb{W}^{(q) \prime}(a-x) - \beta \tilde{r}^{(q)\prime \prime}_{b-a}(b-x). \label{v_a_b_twice_derivative}
\end{align}
Here, integration by parts gives, for $x \in (0, \infty) \backslash \{a^*, b^*\}$,
\begin{align} \label{r_second_derivative}
\begin{split}
r^{(q) \prime \prime }_{b-a}(b-x) &= q \Big[ W^{(q)\prime}(b-x)+ \delta \Big( \mathbb{W}^{(q)\prime}(a-x)W^{(q)}(b-a)+\int_x^a\mathbb{W}^{(q)\prime}(u-x)W^{(q)\prime}(b-u) \diff u \Big) 1_{\{x < a\}}\Big], \\
\tilde{r}^{(q) \prime \prime}_{b-a}(b-x) &= qW^{(q)}(b-x)+\delta \Big( \mathbb{W}^{(q)\prime}(a-x)Z^{(q)}(b-a)+\int_x^a\mathbb{W}^{(q)\prime}(u-x)Z^{(q)\prime}(b-u) \diff u \Big) 1_{\{x < a\}},
\end{split}
\end{align}
and, in particular, 
\begin{align}
r^{(q) \prime \prime}_{b-a}(0+) &= qW^{(q) \prime}(0+), \quad \tilde{r}^{(q) \prime \prime}_{b-a}(0+) = q W^{(q)}(0) =0, \label{r_twice_der_zero}
\end{align}
and, for all $b > a \geq 0$, 
\begin{align}\label{r_twice_der_a_b}
\begin{split}
r^{(q) \prime \prime}_{b-a}((b-a)+) - r^{(q) \prime \prime}_{b-a}((b-a)-) &= q \delta \mathbb{W}^{(q)\prime}(0+)W^{(q)}(b-a), \\
\tilde{r}^{(q) \prime \prime}_{b-a}((b-a)+) - \tilde{r}^{(q) \prime \prime}_{b-a}((b-a)-) &= \delta  \mathbb{W}^{(q)\prime}(0+)Z^{(q)}(b-a).
\end{split}
\end{align}

Using these, we pursue the pair $(a,b)$ such that the $v_{a,b}$ becomes smoother at $a$ and $b$.

\begin{description}
\item[Smoothness at $b$:] 
By \eqref{v_a_b_derivative} and \eqref{r_prime_zero},
\begin{align*}
v_{a,b}'(b-)=\frac{W^{(q)}(0)}{r_{b-a}^{(q)}(b)}\Gamma(a,b) +\beta \quad \textrm{and} \quad v_{a,b}'(b+) = \beta.
\end{align*}
By this and \eqref{eq:Wqp0}, the continuous differentiability holds (i.e., $v_{a,b}'(b-) = v_{a,b}'(b+) = \beta$) for any choice of $(a,b)$ for the unbounded variation case while, in the bounded variation case, it holds if and only if
	\begin{equation}
	\label{scsp_2}
	\mathbf{C}_b: \Gamma(a,b) =0
	\end{equation}
	holds.
For the unbounded variation case we further pursue twice continuous differentiability.   By \eqref{v_a_b_twice_derivative} and \eqref{r_twice_der_zero},
\begin{align*}
v_{a,b}''(b-)= -\frac{\Gamma(a,b)}{r^{(q)}_{b-a}(b)}W^{(q)\prime}(0+) 
\quad \textrm{and} \quad v_{a,b}''(b+) = 0,
	\end{align*}
and hence the twice continuous differentiability holds (i.e., $v_{a,b}''(b-) = v_{a,b}''(b+) = 0$) if and only if \eqref{scsp_2} holds.

\item[Smoothness at $a$:]  Regarding the differentiability at $a$, if $a > 0$, by \eqref{v_a_b_derivative} and \eqref{r_prime_b_a}, 
\begin{align*}
v_{a,b}'(a-) - v_{a,b}'(a+)&=   \delta \mathbb{W}^{(q)}(0)  \Big[ {\Gamma(a,b)}  \frac {W^{(q)}(b-a)  } {r^{(q)}_{b-a}(b)} - \beta \gamma(a,b) \Big], 
\end{align*}
where we define, for $0 \leq a < b$,
	\begin{align}
\gamma(a,b) := \beta^{-1}- Z^{(q)}(b-a). \label{small_gamma}
\end{align}
Hence the continuous differentiability at $a$ holds automatically for the case of unbounded variation, while for the case of bounded variation (by  \eqref{eq:Wqp0}) it holds on condition that 
\begin{align}\label{scsp_1}
	\mathbf{C}_a: \frac{W^{(q)}(b-a)}{r^{(q)}_{b-a}(b)} \Gamma(a,b) - \beta \gamma(a,b)
	=0.
\end{align}

In the unbounded variation case we further pursue twice continuous differentiability.
By \eqref{v_a_b_twice_derivative} and \eqref{r_twice_der_a_b}, if $a > 0$,
\begin{align*}
v_{a,b}''(a-) - v_{a,b}''(a+)&= - \delta \mathbb{W}^{(q) \prime}(0+) \Big[ {\Gamma(a,b)} \frac { W^{(q)}(b-a) } {r^{(q)}_{b-a}(b)} - \beta \gamma(a,b)  \Big].
\end{align*}
Hence, the twice continuous differentiability at $a$ holds if and only if \eqref{scsp_1} holds.

Note that, if \eqref{scsp_2} holds, then condition \eqref{scsp_1}  simplifies to 
	\begin{equation}\label{oc_2}
	\mathbf{C}_a': \gamma(a,b) = 0.
	\end{equation}	
	
\end{description}
	
In summary, we have the following lemma.
\begin{lemma} \label{lemma_smooth_fit}

\begin{enumerate}
\item For $0 \leq a < b$, if condition $\mathbf{C}_b$ as in \eqref{scsp_2} holds, then $v_{a,b}$ is continuously differentiable  (resp.\ twice continuously differentiable) at $b$ when $Y$ has paths of bounded (resp.\ unbounded) variation.
\item For $0 < a < b$, if condition $\mathbf{C}_a$ as in \eqref{scsp_1} holds, then $v_{a,b}$ is continuously differentiable  (resp.\ twice continuously differentiable) at $a$ when $Y$ has paths of bounded (resp.\ unbounded) variation.
\end{enumerate}
\end{lemma}

\subsection{Existence and uniqueness of optimal layer levels $(a^*, b^*)$} \label{subsection_existence}

We shall now pursue our candidate optimal two-layer $(a^*, b^*)$ that simultaneously satisfies conditions $\mathbf{C}_a$ and $\mathbf{C}_b$ (and hence $\mathbf{C}_a'$) when $ a^* > 0$ and $b^* > 0$, respectively.  We will see that $a^* = 0$ and/or $b^* = 0$ can happen and in this case $\mathbf{C}_a$ and/or $\mathbf{C}_b$ do not necessarily hold: these correspond to the persistent refraction and liquidation strategies, respectively. 

The key observation here is that  the functions $\Gamma(\cdot, \cdot)$ and $\gamma(\cdot, \cdot)$ as in \eqref{big_gamma} and \eqref{small_gamma}, respectively, are related by the following: 
\begin{align} 
 \frac \partial {\partial a} \Gamma(a,b) = \delta q \mathbb{W}^{(q)}(a)-q\beta \delta \mathbb{W}^{(q)}(a)Z^{(q)}(b-a) = \delta q \beta  \mathbb{W}^{(q)}(a) \gamma(a,b). \label{relation_gamma_gamma}
\end{align}
Because $\delta q \beta  \mathbb{W}^{(q)}(a)$, for $a > 0$, is positive, if we can find a pair $(a^*,b^*)$ such that the curve $a \mapsto \Gamma(a, b^*)$ gets tangent to the $x$-axis at $a^* > 0$, then necessarily
$\Gamma(a^*,b^*) = \gamma(a^*,b^*) = 0$.  See the plots of Figure \ref{figure_value_function} in Section \ref{numerical_section}.

Now, for each fixed $b \geq 0$, we consider the function 
\begin{align*}
\Gamma(\cdot, b) : a \mapsto \Gamma(a,b), \quad a \in [0,b).
\end{align*}
In view of \eqref{small_gamma}, $\gamma(\cdot, b): a \mapsto \gamma(a,b)$ is monotonically increasing on $[0,b]$ and ends at 
\begin{align*}
\gamma(b-,b)  = \beta^{-1} - 1 > 0.\end{align*}
Hence, in view of the relation \eqref{relation_gamma_gamma}, $\Gamma(\cdot, b)$ on $[0,b]$ is either (1) decreasing and then increasing or (2) increasing.  This implies that its minimum
\begin{align}
\underline{\Gamma} (b) := \min_{0 \leq a \leq b} \Gamma(a, b), \label{gamma_minimum}
\end{align}
is attained at, say $a(b)$, which is, for the case (1), the local minimizer $\hat{a}$ such that $\gamma(\hat{a},b) = 0$ or, for the case (2), zero.  On the other hand, because
\begin{equation}  \label{Gamma_derivative_b}
\frac \partial {\partial b} \Gamma(a,b) = -q \beta r^{(q)}_{b-a}(b)
< 0,  \quad 0 \leq a < b,
\end{equation}
we must have that $\underline{\Gamma} (b)$ is monotonically (strictly) decreasing and continuous in $b \geq 0$. 

In view of these, we first start at $b=0$ and increase $b$ until we attain the desired pair $(a^*, b^*)$.  More precisely, we first compute for the case $b = 0$:
\begin{align*}
\underline{\Gamma} (0) = \Gamma(0,0) &= \delta -q \rho -\beta  \psi'_X(0+). \end{align*}

(i) For the case $\underline{\Gamma} (0)  \leq 0$, we set $a^*=b^* = 0$.  

(ii) Otherwise we shall increase the value of $b$ until we get $b^*$ such that $\underline{\Gamma}(b^*) = 0$.  Because \eqref{Gamma_derivative_b} gives a bound:
\begin{align*}
\max_{0 \leq a \leq b} \frac \partial {\partial b} \Gamma(a,b) & \leq -q \beta < 0,
\end{align*}
it is clear that $\underline{\Gamma}(b) \xrightarrow{b \uparrow \infty} -\infty$.
  Hence, there are two scenarios:
\begin{enumerate}
\item[(ii-1)] $\underline{\Gamma}(b^*) =\min_{0 \leq a \leq b^*} \Gamma(a,b^*) = 0$ is attained at zero: then we set $a^* = 0$;
\item[(ii-2)] $\underline{\Gamma}(b^*) =\min_{0 \leq a \leq b^*} \Gamma(a,b^*)= 0$ is attained at a local minimum, which we call $a^*$.
\end{enumerate}
In both cases (ii-1) and (ii-2), the pair $(a^*, b^*)$ satisfy $\mathbf{C}_b$.  In addition,  for (ii-1), we must have that $\Gamma(\cdot, b^*)$ is increasing (or $\gamma(\cdot, b^*) \geq 0$) on $[0, b^*]$ and hence 
\begin{align*}
\gamma(a^*, b^*) = \gamma(0,b^*) 
= \beta^{-1} - Z^{(q)}(b^*) \geq 0.
\end{align*}For (ii-2), $\mathbf{C}_a$ (and hence $\mathbf{C}_a'$ as well) holds.

We summarize these in the following lemma.
\begin{lemma} \label{lemma_existence}
There exist a  pair $(a^*, b^*)$ such that one of the following holds.
\begin{enumerate}
\item[(i)] $a^*=b^* = 0$ and $\underline{\Gamma}(0)= \delta -q \rho  -\beta \psi'_X(0+) \leq 0$.
\item[(ii-1)] $a^* = 0 < b^*$ such that $\Gamma(a^*, b^*) = 0$ and $\beta^{-1}- Z^{(q)}(b^*) \geq 0$, and $\underline{\Gamma}(0)>0$.
\item[(ii-2)] $0 < a^* < b^*$ such that $\Gamma(a^*, b^*) = \gamma(a^*, b^*)  = 0$, and $\underline{\Gamma}(0)>0$.
\end{enumerate}
\end{lemma}

In fact, the criteria for (ii-1) and (ii-2) can be concisely given as follows.  When $\underline{\Gamma}(0) > 0$, we can define the unique root of $\Gamma(0, b_0) = 0$, which is explicitly written:
\begin{align*}
b_0 := (\overline{Z}^{(q)})^{-1} \Big(\frac {\delta -q \rho } {q \beta}- \frac {\psi'_X(0+)} q \Big).
\end{align*}
By construction, $a^* > 0$ if and only if $\underline{\Gamma}(b_0) < 0$.  This holds
 if and only if $\Gamma(\cdot, b_0)$ attains a local minimum---$\gamma(\hat{a}, b_0)  = \beta^{-1} - Z^{(q)} (b_0 - \hat{a}) = 0$ for some $\hat{a} \in (0, b_0)$, or equivalently $b_0 > (Z^{(q)})^{-1} (\beta^{-1})$. In sum, we have the following.

\begin{remark}  \label{remark_b_0}  When $\underline{\Gamma} (0) > 0$, (ii-1) holds if $b_0 \leq (Z^{(q)})^{-1} (\beta^{-1})$ and (ii-2) holds otherwise.  For both cases, we have $b_0 \wedge (Z^{(q)})^{-1} (\beta^{-1}) \leq b^* \leq b_0$.
\end{remark}

\begin{remark}\label{R_values}
Interestingly, in the formulation \eqref{E_Obj1}, the condition expressed in case (i) in Lemma \ref{lemma_existence} can be rewritten as
\begin{equation}\label{liqcondition2} \beta_A\frac{\delta}{q}+\beta_S \frac{\mu_X}{q} \leq \tilde{\rho}, \end{equation}
where  $\mu_X=-\psi_X'(0+)$ is the drift of the process $X$ (after payment of continuous dividends). Equation \eqref{liqcondition2} considers the expected present value of net dividends coming from refractive and reflective dividends (respectively), as opposed to the terminal payoff/penalty. 
\end{remark}

\section{Optimality of the two-layer $(a^*,b^*)$ strategy} \label{S_opti}
For the selected layer levels $(a^*, b^*)$, the form of $v_{a^*, b^*}$ can be written concisely.
For $b^* > 0$ (i.e.\ the case (ii)), because $\mathbf{C}_b$ as in \eqref{scsp_2} holds, \eqref{v_a_b_2} simplifies to
\begin{equation}  \label{v_a_b_formula}
v_{a^*,b^*}(x) = 
\frac{\delta}{q}  \mathbb{Z}^{(q)}(a^*-x)-\beta \tilde{r}^{(q)}_{b^*-a^*}(b^*-x), \quad  x \geq 0,	\end{equation}
	where in particular
	\begin{align}
	v_{a^*,b^*}(x)&=\frac{\delta}{q}-\beta R^{(q)} (b^*-x), \quad a^* \leq  x, \label{value_bw_a_b} \\
	v_{a^*,b^*}(x) &= \beta \Big( x-b^* - \frac {\psi'_X (0+)} q \Big) +\frac{\delta}{q}, \quad x \geq b^*. \label{value_over_b}
	\end{align}
On the other hand, for $a^* = b^* = 0$ (i.e.\ the case (i)), 
		\begin{equation} \label{v_0_0}
		v_{0,0}(x)= \beta x+ \rho, \quad x \geq 0.
	\end{equation}
For both cases, it can be confirmed that 
\begin{align}
v_{a^*,b^*}(0) = \lim_{x \downarrow 0} v_{a^*, b^*}(x) = \rho. \label{limit_at_zero}
\end{align}

We shall now show the optimality of the two-layer $(a^*, b^*)$ strategy.  To this end, we first obtain the verification lemma and then show that the candidate value function $v_{a^*, b^*}$ solves the required variational inequalities.

\subsection{Verification lemma}

We call a measurable function $g$ \textit{sufficiently smooth} if $g$ is $C^1 (0,\infty)$ (resp.\ $C^2 (0,\infty)$) when $X$ has paths of bounded (resp.\ unbounded) variation.
We let $\mathcal{L}_Y$ be the operator for $Y$ acting on a sufficiently smooth function $g$, defined by
\begin{align}
\mathcal{L}_Y g(x)&:= -\gamma_Y g'(x)+\frac{\sigma^2}{2}g''(x) +\int_{(0,\infty)}[g(x + z)-g(x)-g'(x)z\mathbf{1}_{\{0<z<1\}}]\Pi(\mathrm{d}z). \label{generator_Y}
\end{align}

\begin{lemma}[Verification lemma]
	\label{verificationlemma}
	Suppose $\hat{\pi}$ is an admissible dividend strategy such that $v_{\hat{\pi}}$ is sufficiently smooth on $(0,\infty)$, and satisfies
	\begin{align}\label{HJB-inequality}
	\begin{split}
		\sup_{0\leq r\leq\delta} \big((\mathcal{L}_Y - q)v_{\hat{\pi}}(x)-rv'_{\hat{\pi}}(x)+r \big) &\leq 0, \quad x > 0, \\
		v'_{\hat{\pi}}(x)&\geq\beta, \quad x > 0.
		\end{split}
	\end{align} 
	In addition suppose $\rho = v_{\hat{\pi}}(0) \leq \lim_{x \downarrow 0}v_{\hat{\pi}}(x) $. 
	Then $v_{\hat{\pi}}(x)=v(x)$ for all $x\geq0$ and hence $\hat{\pi}$ is an optimal strategy.
\end{lemma}
\begin{proof} 
	By the definition of $v$ as a supremum, it follows that $v_{\hat{\pi}}(x)\leq v(x)$ for all $x\geq0$. We write $w:=v_{\hat{\pi}}$ and show that $w(x)\geq v_\pi(x)$ for all $\pi\in\mathcal{A}$ and $x\geq0$. Fix $\pi\in \mathcal{A}$.
	
	Let $(T_n)_{n\in\mathbb{N}}$ be the sequence of stopping times defined by $T_n :=\inf\left\{t>0:{U}^\pi_t>n\text{ or } U^{\pi}_t<\frac{1}{n} \right\}$. 
	Since ${U}^\pi$ is a semi-martingale and $w$ is sufficiently smooth on $(0, \infty)$ by assumption, we can use the change of variables/It\^o's formula \citep*[see Theorems II.31 and II.32 of][]{Pro05}
to the stopped process $(\mathrm{e}^{-q(t\wedge T_n)}w({U}^\pi_{t\wedge T_n}); t \geq 0)$ to deduce under $\mathbb{P}_x$ that
	\begin{equation*}
		\label{impulse_verif_1}
		\begin{split}
			\mathrm{e}^{-q(t\wedge T_n)}w({U}^\pi_{t\wedge T_n})-w(x)
			= & -\int_{0}^{t\wedge T_n}\mathrm{e}^{-qs} q w({U}^\pi_{s-}) \mathrm{d}s
			+\int_{[0, t\wedge T_n]}\mathrm{e}^{-qs}w'({U}^\pi_{s-}) \mathrm{d}  ( Y_s- {A}^\pi_s  )  \\
			&-\int_0^{t\wedge T_n}\mathrm{e}^{-qs}w'({U}^\pi_{s-}) \mathrm{d} S^{\pi,c}_s - \sum_{0 \leq s\leq t\wedge T_n}\mathrm{e}^{-qs}\Delta w({U}^\pi_{s-}+\Delta S^\pi_s) \\+ &\frac{\sigma^2}{2}\int_0^{t\wedge T_n}\mathrm{e}^{-qs}w''({U}^\pi_{s-})\mathrm{d}s 
			 + \sum_{0 \leq s\leq t\wedge T_n}\mathrm{e}^{-qs}[\Delta w({U}^\pi_{s-}+\Delta Y_s)-w'({U}^\pi_{s-})  \Delta Y_s  ],
		\end{split}
	\end{equation*}
	where we use the following notation: $\Delta w(\zeta_s):=w(\zeta_s)-w(\zeta_{s-})$ and $\zeta^c$ for the continuous part for any process $\zeta$.
	Rewriting the above equation leads to 
	\begin{multline*}
			\mathrm{e}^{-q(t\wedge T_n)}w({U}^\pi_{t\wedge T_n})  -w(x)
			=  \int_{0}^{t\wedge T_n}\mathrm{e}^{-qs}   (\mathcal{L}_Y-q)w({U}^\pi_{s-})   \mathrm{d}s \\
			-\int_{0}^{t\wedge T_n}\mathrm{e}^{-qs}w'({U}^\pi_{s-})\mathrm{d}{A}^\pi_s  -\int_0^{t\wedge T_n}\mathrm{e}^{-qs}w'({U}^\pi_{s-}) \mathrm{d} S^{\pi,c}_s - \sum_{0 \leq s\leq t\wedge T_n}\mathrm{e}^{-qs}\Delta w({U}^\pi_{s-}+\Delta S^\pi_s)  + M_{t \wedge T_n}
	\end{multline*}
	where 
	\begin{align}\label{def_M_martingale}
		\begin{split}
			M_t &:= \int_0^t \sigma  \mathrm{e}^{-qs} w'(U_{s-}^{\pi}) \diff B_s +\lim_{\varepsilon\downarrow 0}\int_{[0,t]} \int_{(\varepsilon,1)}  \mathrm{e}^{-qs}w'(U_{s-}^{\pi})y (N(\diff s\times \diff y)-\Pi(\diff y) \diff s)\\
			&+\int_{[0,t]} \int_{(0,\infty)} \mathrm{e}^{-qs}(w(U_{s-}^\pi+y)-w(U_{s-}^\pi)-w'(U_{s-}^\pi)y\mathbf{1}_{\{y\in (0, 1)\}})(N(\diff s\times \diff y)-\Pi(\diff y) \diff s), \quad t \geq 0, 
		\end{split}
	\end{align}
	where $( B_s; s \geq 0 )$ is a standard Brownian motion and $N$ is a Poisson random measure in   the measure space  $([0,\infty)\times (0, \infty),\B [0,\infty)\times \B (0, \infty), \diff s \times \Pi( \diff x))$.
	\par Hence we derive that
	\begin{equation*}
		\begin{split}
			w(x) = &
			-\int_{0}^{t\wedge T_n}\mathrm{e}^{-qs}  \left[ (\mathcal{L}_Y-q)w({U}^\pi_{s-})-a^\pi_sw'({U}^\pi_{s-})+ a^\pi_s \right]  \mathrm{d}s
			\\
			&+\int_0^{t\wedge T_n}\mathrm{e}^{-qs}w'({U}^\pi_{s-}) \mathrm{d} S^{\pi,c}_s + \sum_{0 \leq s\leq t\wedge T_n}\mathrm{e}^{-qs}\Delta w({U}^\pi_{s-}+\Delta S^\pi_s) \\
			& + \int_{0}^{t\wedge T_n}\mathrm{e}^{-qs} a^\pi_s \mathrm{d}s - M_{t\wedge T_n} + \mathrm{e}^{-q(t\wedge T_n)}w({U}^\pi_{t\wedge T_n}).
		\end{split}
	\end{equation*}
	Using  the assumption \eqref{HJB-inequality} and $a_s^\pi \in [0, \delta]$ a.s.\ for all $s \geq 0$, we have
	\begin{equation} \label{w_lower}
	\begin{split}
	w(x) \geq &
	\int_{0}^{t\wedge T_n}\mathrm{e}^{-qs} a^\pi_s \mathrm{d}s +\beta\int_{[0, t\wedge T_n]}\mathrm{e}^{-qs}\mathrm{d} S^\pi_s - M_{t\wedge T_n} + \mathrm{e}^{-q(t\wedge T_n)}w({U}^\pi_{t\wedge T_n}).
	\end{split}
	\end{equation} 
		In addition by the compensation formula \citep*[see Corollary 4.6 of][]{Kyp06}, $(M_{t \wedge T_n}:t\geq0 )$ is a zero-mean $\mathbb{P}_x$-martingale.  
	On the other hand,
	\begin{align*} 
	\E_x \big( \mathrm{e}^{-q(t\wedge T_n)}w({U}^\pi_{t\wedge T_n}) \big) = \E_x \big( 1_{\{ t < T_n\}}\mathrm{e}^{-q t}w({U}^\pi_{t}) \big) + \E_x \big( 1_{\{ t > T_n\}} \mathrm{e}^{-qT_n}w({U}^\pi_{T_n}) \big),
	\end{align*}
	where the former vanishes as $t \rightarrow \infty$ because $w$ is bounded on $[0, n]$. 
	For the latter for sufficiently large $n$ such that $w(n) > 0$ (note that $w(\infty) = \infty$), we have, because $\rho = w(0) \leq \lim_{x \downarrow 0} w(x)$, $x \geq 0$, by assumption,
	\begin{align*}
	\E_x \big( 1_{\{ t > T_n\}} \mathrm{e}^{-qT_n}w({U}^\pi_{T_n}) \big) &= \E_x \big( 1_{\{ t > T_n, U^\pi_{T_n} \leq 1 / n \}} \mathrm{e}^{-qT_n}w( U^\pi_{T_n}) \big) + \E_x \big( 1_{\{ t > T_n, U^\pi_{T_n} >  n  \}} \mathrm{e}^{-qT_n}w({U}^\pi_{T_n}) \big) \\
	&\geq \E_x \big( 1_{\{ t > T_n, U^\pi_{T_n} \leq  1 / n \}} \mathrm{e}^{-qT_n} w({U}^\pi_{T_n}) \big).	
\end{align*}
By dominated convergence and because $T_n$ converges to $\sigma^{\pi}$ a.s.,
we have
\begin{align*}
\lim_{t, n \rightarrow \infty}\E_x \big( 1_{\{ t > T_n, U^\pi_{T_n}  \leq  1 / n \}} \mathrm{e}^{-qT_n} w({U}^\pi_{T_n})\big) 
\geq\rho \E (e^{-q \sigma^\pi}).
\end{align*}
In sum,
	\begin{align} 
	\liminf_{t,n \rightarrow \infty}\E_x \big( \mathrm{e}^{-q(t\wedge T_n)}w({U}^\pi_{t\wedge T_n}) \big) \geq \rho \E (e^{-q \sigma^\pi}). \label{rho_lower_bound}
	\end{align}
	
	\par Now  taking expectations  in \eqref{w_lower} and
	letting $t$ and $n$ go to infinity ($T_n\nearrow\infty$ $\mathbb{P}_x$-a.s.), the monotone convergence theorem and \eqref{rho_lower_bound} give 
	\begin{equation*}
		w(x) \geq \mathbb{E}_x \left( \int_{0}^{\sigma^\pi}\mathrm{e}^{-qs} a^\pi_s \mathrm{d}s+\beta\int_{[0,\sigma^\pi]}\mathrm{e}^{-qs}\mathrm{d} S^\pi_s +\rho e^{-q \sigma^\pi} \right) =v_\pi(x).
	\end{equation*}
	This completes the proof.
	\end{proof}

\subsection{Proof of optimality}
We shall now show that the function $v_{a^*, b^*}$ satisfies all the requirements given in Lemma \ref{verificationlemma}.  
We shall first see that the function $v_{a^*, b^*}$ is concave with its slope $v_{a^*, b^*}'(a^*) = 1$ and $v_{a^*, b^*}'(b^*) = \beta$, respectively, 
when $a^* > 0$ and $b^* > 0$.

\begin{lemma}\label{ver_1} The function $v_{a^*, b^*}$ is concave and the following holds:
\begin{enumerate}
	\item  For $x>a^*$, we have $\beta \leq v_{a^*,b^*}'(x)\leq 1$;
	\item  For $0 < x<a^*$, we have $v_{a^*,b^*}'(x)\geq 1 > \beta$.
\end{enumerate}
\end{lemma}
\begin{proof}

The case $b^* = 0$ holds trivially in view of 
\eqref{v_0_0}. Hence, we focus on the case $b^* > 0$. 
By \eqref{v_a_b_derivative} and \eqref{scsp_2},
	\begin{align*}
v_{a^*,b^*}'(x)&=  - \delta \mathbb{W}^{(q)}(a^*-x) + \beta \tilde{r}^{(q)\prime}_{b^*-a^*}(b^*-x), \quad x > 0,
\end{align*}
and, differentiating it,
	\begin{align*}
v_{a^*,b^*}''(x)&=   \delta \mathbb{W}^{(q) \prime}(a^*-x) - \beta \tilde{r}^{(q)\prime \prime}_{b^*-a^*}(b^*-x), \quad x \in (0, \infty) \backslash \{a^*, b^*\}.
\end{align*}
In particular, for $x>a^*$, by \eqref{r_second_derivative},
\begin{equation}\label{sdvf_2}
v_{a^*,b^*}''(x)= -q\beta W^{(q)}(b^*-x) \leq 0,
\end{equation}
and, if $a^* > 0$ (which means $\mathbf{C}_a'$ is satisfied) for $x<a^*$, by \eqref{r_second_derivative} and \eqref{oc_2}, 
\begin{align}\label{sdvf_1}
\begin{split}
v_{a^*,b^*}''(x)
&=-\beta\Big(qW^{(q)}(b^*-x)+\delta\int_x^{a^*}\mathbb{W}^{(q)\prime}(u-x)Z^{(q)\prime}(b^*-u) \diff u\Big)\leq 0.
\end{split}
\end{align}

This shows that $v_{a^*, b^*}$ is concave on $(0, b^*)$.  This can be extended to $(0, \infty)$ by
the continuous differentiability at $b^*$ as in Lemma \ref{lemma_smooth_fit} and $v_{a^*,b^*}'(x)=\beta$ for $x > b^*$.

To show (1) and (2), by the obtained concavity and $v_{a^*,b^*}(b^*) = \beta$, it is left to show for the case $a^* > 0$ that $v_{a^*,b^*}'(a^*)=1$ and for the case $a^* = 0$ that $v_{0,b^*}'(0+) \leq 1$.  Indeed, for the former, by (\ref{oc_2}), $v_{a^*,b^*}'(a^*)=\beta Z^{(q)}(b^*-a^*)=1$.
For the latter, by Lemma \ref{lemma_existence} (ii-1), $v_{0,b^*}'(0+) =\beta Z^{(q)}(b^*) \leq 1$.
\end{proof}

We shall now compute the generator terms appearing in the variational inequalities. Toward this end, we use the following lemma, which is a direct consequence of the well-known results that can be shown by the martingale arguments.
Similarly to $\mathcal{L}_Y$ as in \eqref{generator_Y}, let  $\mathcal{L}_X$ be the operator for $X$ acting on sufficiently smooth function $g$, defined by
\begin{equation*}
\begin{split}
\mathcal{L}_X g(x)&:= \mathcal{L}_Y g(x) - \delta g'(x) = -\gamma_X g'(x)+\frac{\sigma^2}{2}g''(x) +\int_{(0,\infty)}[g(x+z)-g(x) -g'(x)z\mathbf{1}_{\{0<z<1\}}]\Pi(\mathrm{d}z).
\end{split}
\end{equation*}

\begin{lemma} \label{remark_harmonic} 
(i) For $x < a^*$,
$(\mathcal{L}_Y -q) \mathbb{Z}^{(q)}(a^*-x) = 0$.  (ii) For $x < b^*$, $ (\mathcal{L}_X-q)R^{(q)} (b^*-x)= 0$. 
(iii) For $x < \bar{b} \leq b^*$,
$(\mathcal{L}_Y-q) \big( \int_x^{\bar{b}}\mathbb{W}^{(q)}(u-x)Z^{(q)}(b^*-u)\diff u \big) = Z^{(q)} (b^*-x)$.
\end{lemma}
\begin{proof} (i) and (ii) hold immediately by, for instance, the proof of Theorem 2.1 in \citet*{BaKyYa12}.  (iii) holds by applying the result in the proof of Lemma 4.5 in \citet*{EgYa13}.
\end{proof}

Using this, we have the following lemma.
\begin{lemma}\label{ver_2} 
	\begin{itemize}
		\item[(i)] Suppose $a^* > 0$. For $0 <  x<a^*$, we have $(\mathcal{L}_Y-q)v_{a^*,b^*}(x)=0$.
		\item [(ii)] Suppose $b^* > 0$. For $a^*<x<b^*$, we have $(\mathcal{L}_X-q)v_{a^*,b^*}(x)+\delta=0$.
		\item[(iii)] For $x>b^*$, $(\mathcal{L}_X-q)v_{a^*,b^*}(x)+ \delta\leq 0$.
	\end{itemize}
\end{lemma}
\begin{proof}

\par (i) 
We first rewrite \eqref{v_a_b_formula} as
\begin{align*}
v_{a^*,b^*}(x)=\frac {\delta \mathbb{Z}^{(q)}(a^*-x)} q -\beta \Big(R^{(q)} (b^*-x)+\delta\int_{x}^{a^*}\mathbb{W}^{(q)}(u-x)Z^{(q)}(b^*-u)\ud u\Big).
\end{align*}
We notice that by Lemma \ref{remark_harmonic} (ii)
\begin{align*}
	&(\mathcal{L}_Y-q) R^{(q)} (b^*-x)  = - \delta Z^{(q)}(b^*-x). 
\end{align*}
Recall also Lemma \ref{remark_harmonic} (i) and (iii).
Putting the pieces together, we have the result.

(ii) In view of the form \eqref{value_bw_a_b},
 Lemma \ref{remark_harmonic} (ii) shows the result.
\par (iii) When $b^* > 0$, by \eqref{value_over_b}, for $x > b^*$,
\begin{align*}
	(\mathcal{L}_X-q) v_{a^*,b^*}(x)
	&=-\beta \psi'_X(0+)-\beta q(x-b^*)-\delta+\beta \psi'_X(0+)\\
	&=-\delta -\beta q(x-b^*) \leq -  \delta.
\end{align*}

When $b^* = 0$, by Lemma \ref{lemma_existence} (i), $\delta -q \rho  -\beta  \psi'_X(0+) \leq 0$ and hence, for $x > 0$,
\begin{align*}
	(\mathcal{L}_X-q)v_{0,0}(x)&=-\beta \psi'_X(0+)- (\beta qx +  q \rho )\leq - \delta.
\end{align*}

\end{proof}

	Using the above lemmas, we now confirm that $v_{a^*,b^*}$ satisfies the variational inequalities \eqref{HJB-inequality}. First, proceeding like in the proof of Lemma 6  in \citet*{KyLoPe12} \citep*[see also the proof of Lemma 4.3 of][]{HePeYa16},  Lemmas \ref{ver_1} and \ref{ver_2} imply the first item of \eqref{HJB-inequality}.  The second item of \eqref{HJB-inequality} is immediate by Lemma \ref{ver_1}. Finally, \eqref{limit_at_zero} guarantees the last requirement in the lemma. Therefore we have our main result as follows.
	\begin{theorem}  \label{theorem_dividend} The two-layer $(a^*, b^*)$ strategy for $(a^*, b^*)$ selected as in Lemma \ref{lemma_existence} is optimal, and the value function is given by $v(x) = v_{a^*,b^*}(x)$ for all $0 \leq x < \infty$.
	\end{theorem}

\subsection{Convergence results} \label{subsection_convergence}

	We conclude this section with some asymptotic results with respect to the parameters $\beta$ and $\delta$.  Here, we show in the limit that the absolutely continuous control $A^{a^*, b^*}$ and/or the singular control $S^{a^*, b^*}$ will vanish or get persistent; i.e. $(b^* - a^*) \rightarrow 0$, $b^* \rightarrow \infty$ or $a^* \rightarrow 0$.  Motivated by these results, we confirm numerically in the next section the convergence of the optimal solutions to those in  \citet*{BaKyYa13} and \citet*{YiWe13}.

Recall as in Lemma \ref{lemma_existence} that $b^* > 0$ if and only if $\underline{\Gamma} (0) = \delta-q \rho   -\beta \psi'_X(0+) > 0$.
 
 \begin{description}
 \item[(i)] \underline{Limits as $\beta \downarrow 0$ (i.e.\ $\beta_S \downarrow 0$):} Note that $\underline{\Gamma} (0)  \xrightarrow{\beta \downarrow 0} \delta-q \rho$. Here, we are comparing the sizes of $\rho$ (the liquidation value) and $\delta/q$ (the value from the refractive barrier; see also Remark \ref{R_values}). When $\beta$ is approaching 0, it does not make sense to keep paying reflective dividends any more. Therefore, a decision between paying only refractive dividends ($b^* \rightarrow \infty$), and liquidating ($b^*=0$) has to be made. We have:
\begin{enumerate}
\item Suppose $q \rho / \delta < 1$. For sufficiently small $\beta > 0$, $\underline{\Gamma} (0)  > 0$ and hence $b^* > 0$ by Lemma \ref{lemma_existence}.
In view of Remark \ref{remark_b_0}, because both $b_0$ and $(Z^{(q)})^{-1} (\beta^{-1})$ go to infinity as $\beta \downarrow 0$, we must have $b^* \xrightarrow{\beta \downarrow 0} \infty$. In words, the value of refractive dividends is greater than that of liquidating.

\item Suppose $q \rho / \delta > 1$. For sufficiently small $\beta > 0$, we must have $\underline{\Gamma} (0) \leq 0$, and hence $a^* = b^* = 0$. In words, the value of liquidating is greater than that of refractive dividends.

\item Suppose $q \rho / \delta = 1$ and $\psi'_X(0+) < 0$. Then, $\underline{\Gamma} (0) > 0$ and hence $b^* > 0$.
In view of Remark \ref{remark_b_0}, because $b_0 = (\overline{Z}^{(q)})^{-1} (-\psi'_X(0+)/q)$ (which does not depend on $\beta$) and  $(Z^{(q)})^{-1} (\beta^{-1}) \xrightarrow{\beta \downarrow 0} \infty$, we must have that $b^* \xrightarrow{\beta \downarrow 0} (\overline{Z}^{(q)})^{-1} (-\psi'_X(0+)/q)$.  This also shows that $a^* \xrightarrow{\beta \downarrow 0} 0$. In this case, both the values are equal. Furthermore, the drift in the refractive layer $\mu_X = -\psi_X^\prime (0+)$ is greater than 0 (positive drift). Therefore, there is value in trying to harvest those positive prospects. It is then optimal to distribute refractive dividends as much as we can, until ruin, i.e. $a^*=0$.

\item Suppose $q \rho / \delta = 1$ and $\psi'_X(0+) \geq 0$.   Then for $\beta > 0$, we must have $\delta-q \rho- \beta \psi'_X(0+) \leq 0$. Hence, by Lemma \ref{lemma_existence}, $a^*=b^*=0$ for all $\beta>0$. This is similar to case [3.] immediately above, except that we are having a non-positive drift. Unsurprisingly, it is then optimal to liquidate the company immediately, i.e. $a^*=b^*=0$.

\end{enumerate}

\item[(ii)] \underline{Limits as $\beta \uparrow 1$:} 
Note that we have $\underline{\Gamma} (0)  \xrightarrow{\beta \uparrow 1} - \psi_Y'(0+) - q \rho$. Here, we are comparing the sizes of $\rho$ (the liquidation value) and $\mu_Y/q := -\psi_Y'(0+)/q$ (the value from the reflective barrier). When $\beta$ is approaching 1, the two-layer dividend strategy converges to the pure reflective (barrier) strategy when transaction costs are the same \citep*[see, e.g.][]{BaKyYa13}. The value of the reflective payment is $\mu_Y/q$ (since $\beta\rightarrow 1$), and here the four cases boil down to two: whether one should have a positive barrier $b^*$ (case [1.] below), or liquidate immediately (cases [2.-4.] below):
\begin{enumerate}
\item  Suppose $\psi_Y'(0+) + q \rho < 0$.   Then, for sufficiently large $\beta < 1$, Lemma \ref{lemma_existence} guarantees that $b^* > 0$.  We have $(Z^{(q)})^{-1} (\beta^{-1}) \xrightarrow{\beta \uparrow 1} 0$ and
$b_0 \xrightarrow{\beta \uparrow 1} (\overline{Z}^{(q)})^{-1} \big( -(q \rho + \psi'_Y(0+)) / q \big) > 0$. Hence for sufficiently large $\beta < 1$ we have $(Z^{(q)})^{-1} (\beta^{-1}) < b_0$ and hence, by Remark \ref{remark_b_0}, $0 < a^* < b^*$ and therefore $Z^{(q)} (b^*-a^*) = \beta^{-1}$.  This implies $b^* - a^* \xrightarrow{\beta \uparrow 1} 0$.

\item Suppose $\psi_Y'(0+) + q \rho > 0$.   Then, for sufficiently large $\beta < 1$, Lemma \ref{lemma_existence} guarantees that $a^* = b^* = 0$.  

\item Suppose $\psi_Y'(0+) + q \rho = 0$ and $\psi_X'(0+) > 0$.   We have $\partial \underline{\Gamma} (0)  / {\partial \beta} = - \psi'_X(0+) < 0$.
 Hence for $\beta < 1$, we must have $\delta-q \rho - \beta \psi'_X(0+) > 0$.  This means $b^* > 0$. Because $b^* \leq b_0$ and
$b_0 \xrightarrow{\beta \uparrow 1} (\overline{Z}^{(q)})^{-1} \Big(- \rho - \frac {\psi'_Y(0+)} q \Big) = 0$, we have $\lim_{\beta \uparrow 1} a^* = \lim_{\beta \uparrow 1} b^* = 0$.
 
\item Suppose $\psi_Y'(0+) + q \rho = 0$ and $\psi_X'(0+) \leq 0$.   We have $\partial \underline{\Gamma} (0)  / {\partial \beta}  = - \psi'_X(0+) \geq 0$.
 Hence for $\beta < 1$, we must have $\delta-q \rho- \beta \psi'_X(0+) \leq 0$, and  by Lemma  \ref{lemma_existence} that $a^*=b^*=0$ for all $\beta<1$.
\end{enumerate}

\item[(iii)] \underline{Limits as $\delta \uparrow \infty$:}  
We have $\underline{\Gamma} (0)    \xrightarrow{\delta \uparrow \infty} \infty$.  Hence, for sufficiently large $\delta$, we must have $b^* > 0$. 
Note that $Z^{(q)}(x) \xrightarrow{\delta \uparrow \infty} 1$ and hence $(Z^{(q)})^{-1} (\beta^{-1}) \xrightarrow{\delta \uparrow \infty} \infty$. 
 Also, 
 \begin{align*}
b_0 := (\overline{Z}^{(q)})^{-1} \Big(\frac {\delta-q \rho} {q \beta}- \frac {\psi'_X(0+)} q \Big) \xrightarrow{\delta \uparrow \infty} \infty.
\end{align*}
Now in view of Remark \ref{remark_b_0}, $b^* \geq b_0 \wedge (Z^{(q)})^{-1} (\beta^{-1})  \xrightarrow{\delta \uparrow \infty} \infty$. This means that when refractive payouts can be higher ($\delta$ is higher), while still being cheaper in terms of transaction costs, then a pure threshold strategy $a^*$ (with $b^*\rightarrow \infty$) becomes optimal.
\end{description}
		
\section{Numerical Examples} \label{numerical_section}

In this section, we focus on the compound Poisson dual model (with diffusion), where $Y$ and $X$ have i.i.d.\ phase-type distributed jumps. In this case, their scale functions have analytical expressions, and hence the optimal strategy and the value function can be computed instantaneously. The class of processes of this type is important because it can approximate any spectrally one-sided \lev process \citep*[see][]{AsAvPi04,EgYa14}.  Full specifications are given in \ref{S_PTLP}. In Section \ref{S_NI1}, we show how the value function and optimal levels $a^*$ and $b^*$ interact, whereas Section \ref{S_NI2} investigates sensitivity to the main model parameters. In Section \ref{S_NI3} we investigate the impact of volatility of the surplus process on the optimal parameters $a^*$, $b^*$ and the gap between them $b^*-a^*$. For simplicity, in that last section, we focus on a plain vanilla compound Poisson dual model with exponential jumps.

\begin{figure}[htbp]
\begin{center}
\begin{minipage}{1.0\textwidth}
\centering
\begin{tabular}{cc}
 \includegraphics[scale=0.49]{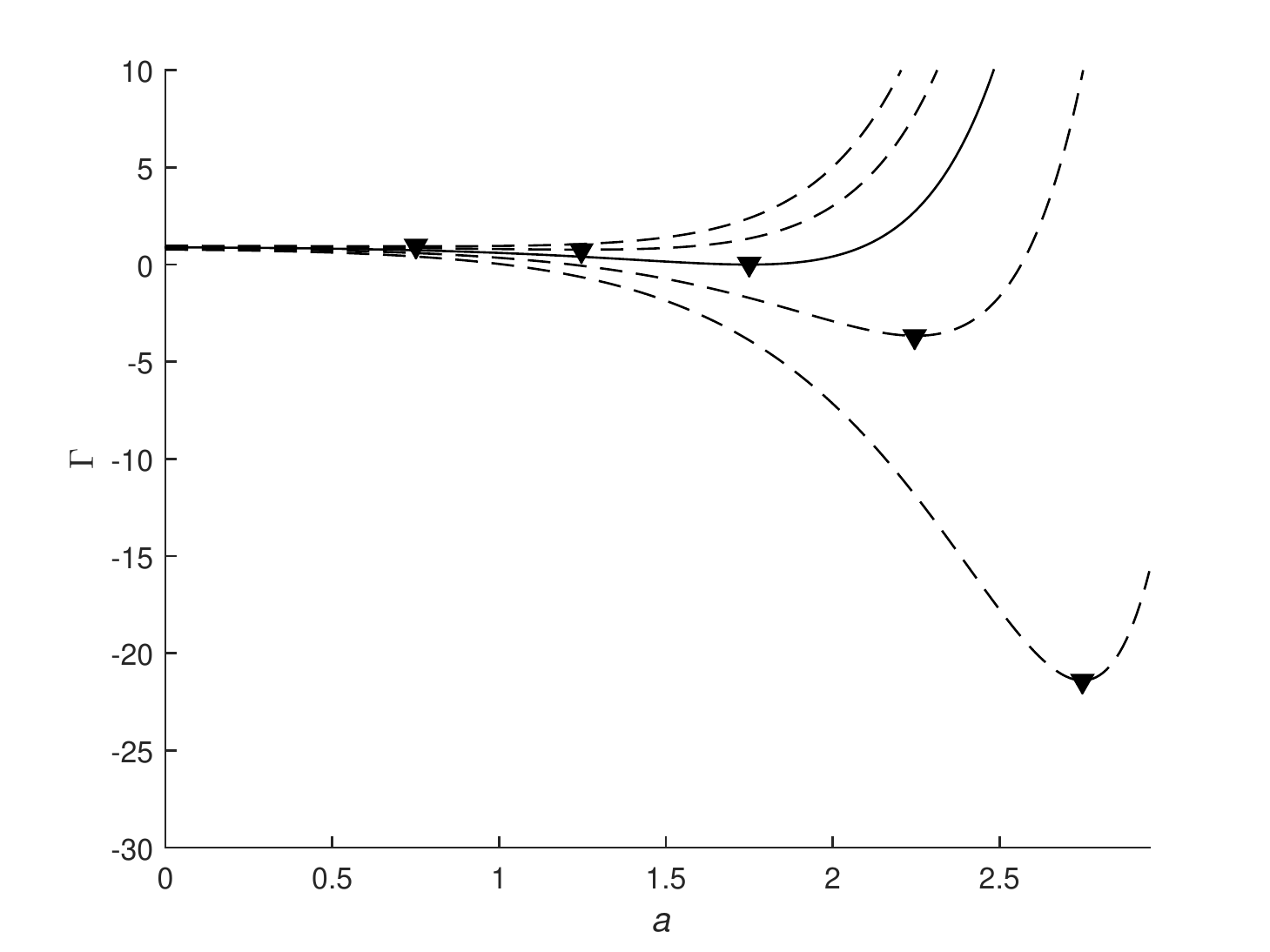} & \includegraphics[scale=0.49]{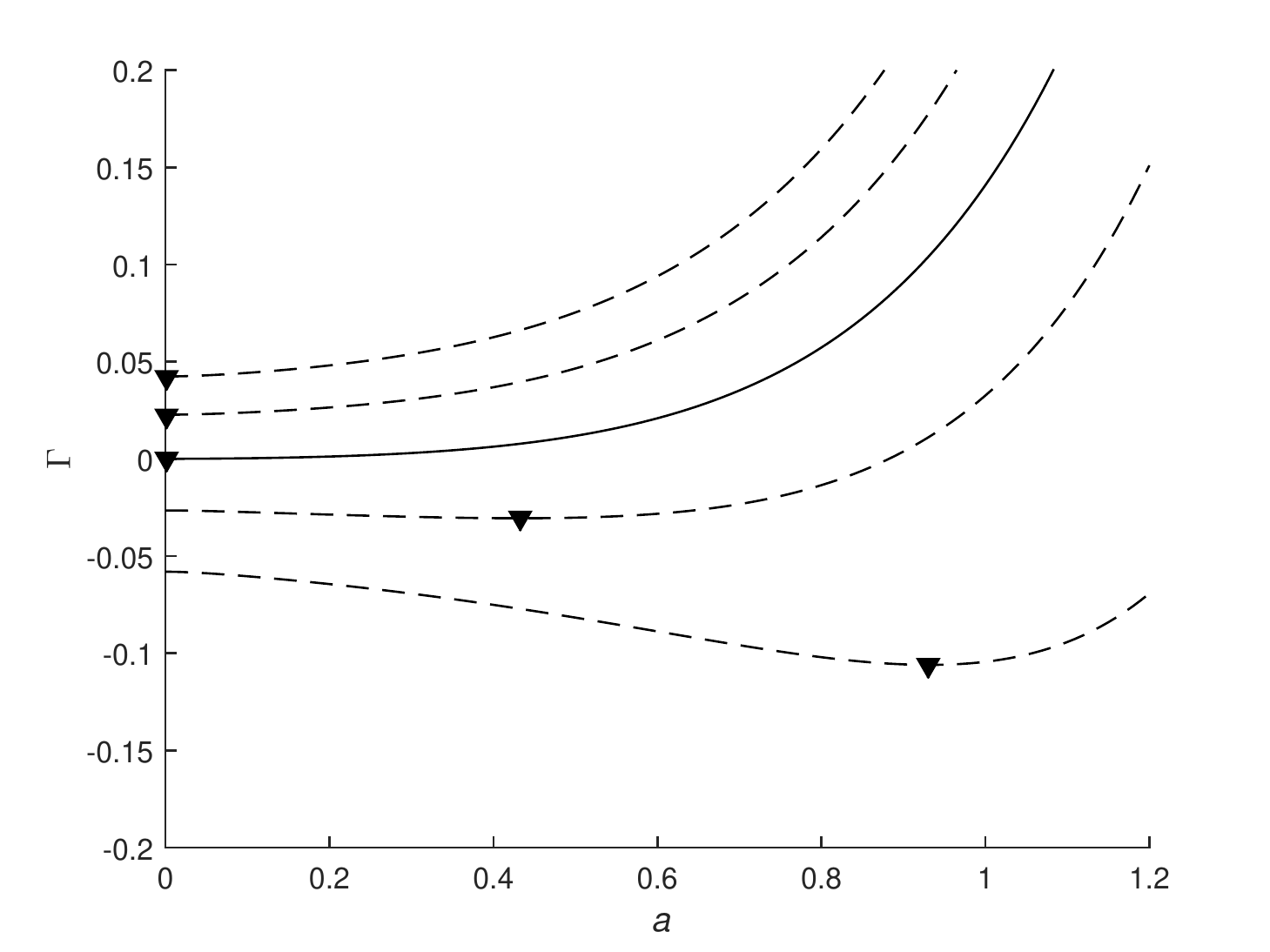}  \\
  \includegraphics[scale=0.49]{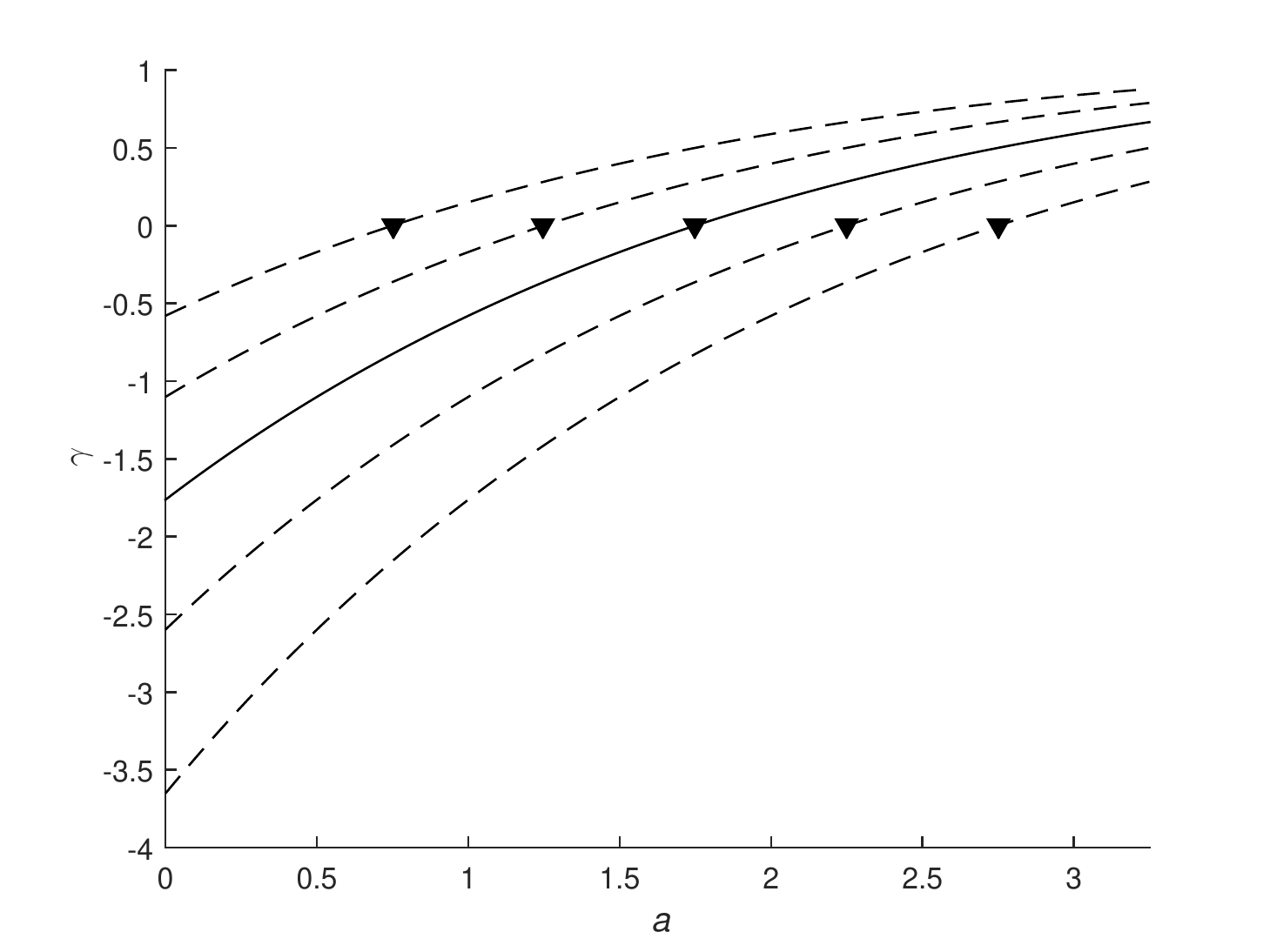} & \includegraphics[scale=0.49]{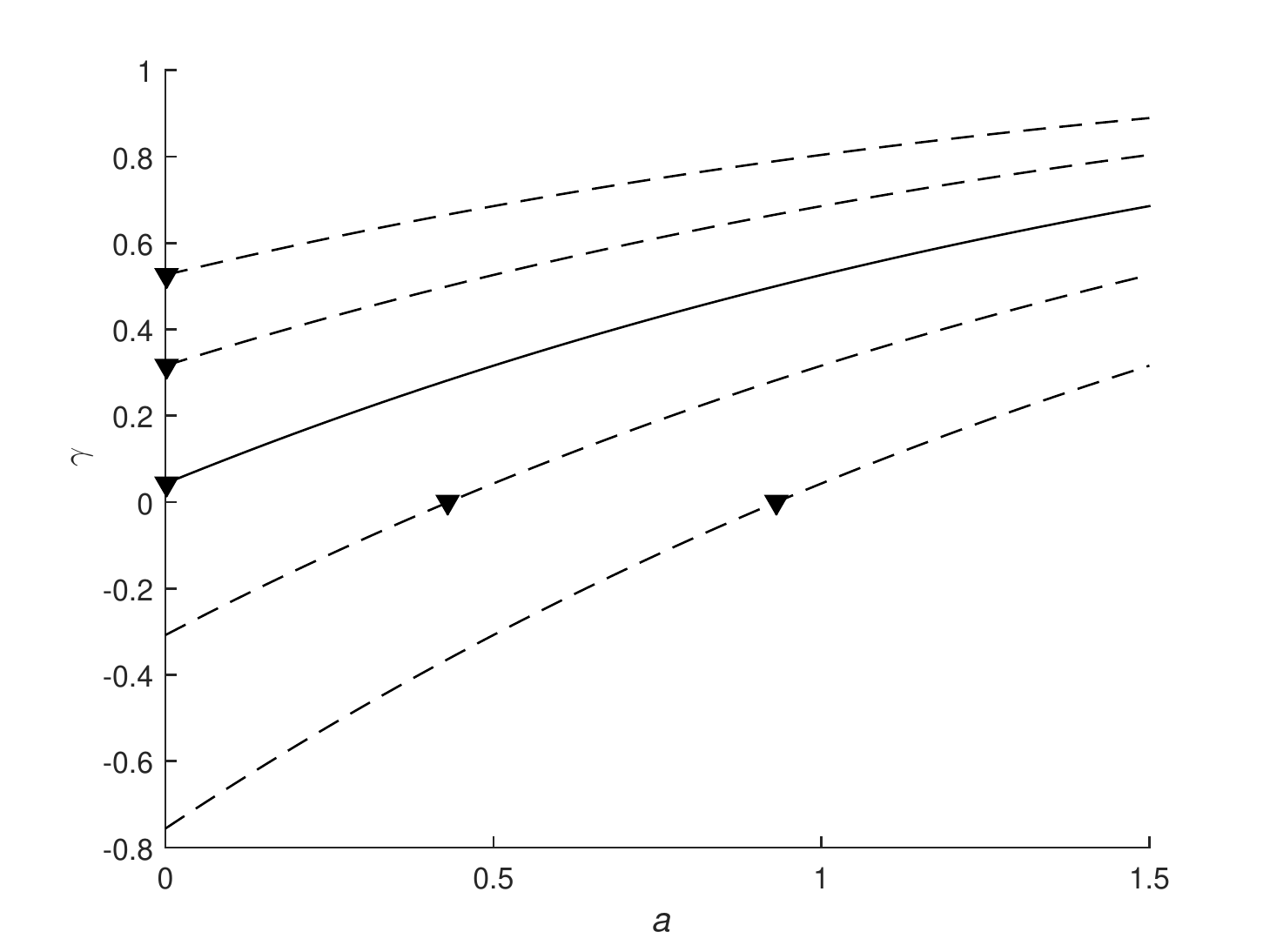}  \\
 \includegraphics[scale=0.49]{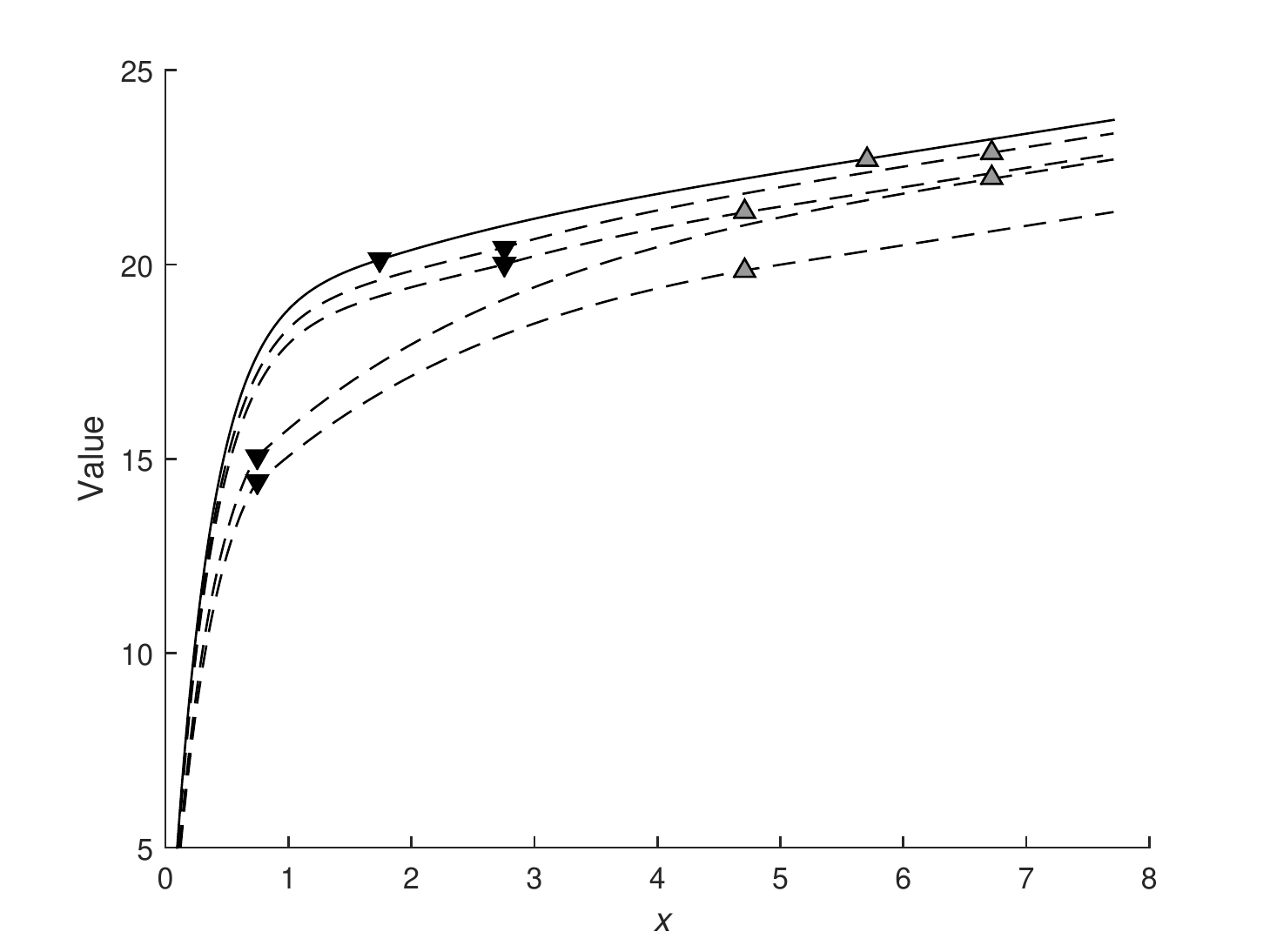} & \includegraphics[scale=0.49]{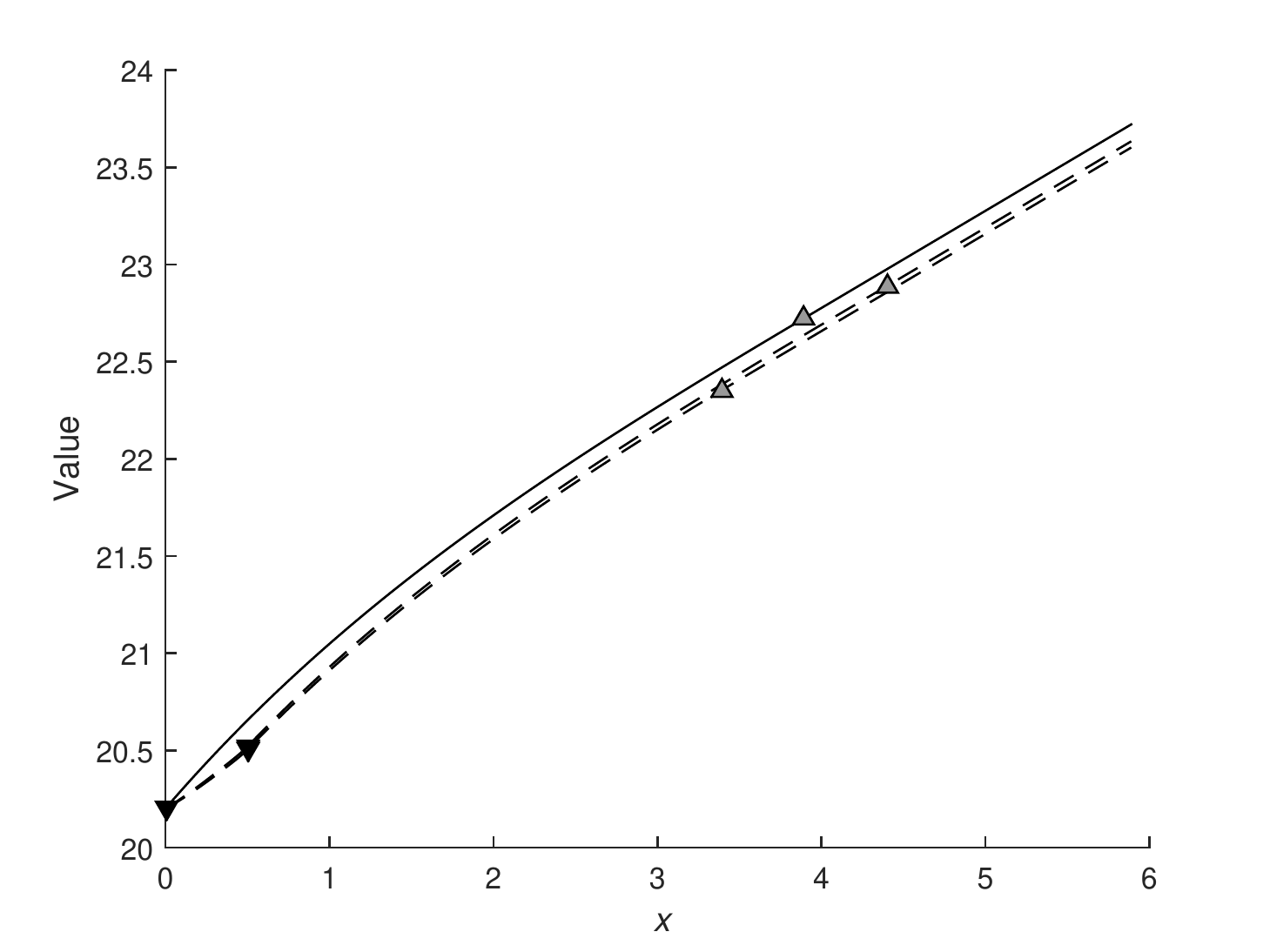} \\
 Case 1 & Case 2 \end{tabular}
\end{minipage}
\caption{(Top row) Plots of $a \mapsto \Gamma(a,b)$, (middle row) $a \mapsto \gamma(a,b)$, and (last row) the value function $v_{a^*, b^*}$ with suboptimal net present values $v_{a, b}$.
} \label{figure_value_function}
\end{center}
\end{figure}

\subsection{Computation of the value function} \label{S_NI1}

We first illustrate the computation of the pair $(a^*, b^*)$ and the value function $v_{a^*,b^*}$.  As in our discussion in Section \ref{subsection_existence}, the function $a \mapsto \gamma(a,b)$ is increasing and  hence the minimizer $a(b)$ of the function $a \mapsto \Gamma(a,b)$ over $[0,b]$ is either positive or zero; for the former case it becomes a local minimum as well. Thus, by a usual bisection method, we can obtain, for $b \geq 0$, $a(b)$ and the infimum function $\underline{\Gamma}(b)$. 
When $\underline{\Gamma}(0) \leq 0$ then $a^* =b^* = 0$, and otherwise $b^*$ becomes the root of $\underline{\Gamma}(b)= 0$.
 The function $\underline{\Gamma}(b)$ is again monotonically decreasing in $b$, and hence another bisection method can be applied to obtain $b^*$.  We then set $a^* = a(b^*)$ and compute the value function $v_{a^*,b^*}$ by the formula \eqref{v_a_b_formula} or \eqref{v_0_0}.

In Figure \ref{figure_value_function} we plot the functions $a \mapsto \gamma(a, b)$, $a \mapsto \Gamma(a, b)$, and $x \mapsto v_{a^*,b^*}(x)$ for Case 1 ($\bar{\rho} := q \rho / \delta = 0$  with $\rho$ the terminal payoff at ruin) where $0 < a^* < b^*$ (left column) and for Case 2 ($\bar{\rho} = 1.01$) where $0 = a^* < b^*$ (right column).  
 In the first and second rows, we plot $\Gamma(\cdot, b)$ and $\gamma(\cdot, b)$, respectively, for $b = b^*-2, b^*-1, b^*, b^*+1, b^*+2$.  The down-pointing triangles indicate the points at $a(b)$. In Case 1, it is confirmed that $a \mapsto \gamma(a,b)$ is increasing and the value of $a^*$ becomes the point at which the curve $a \mapsto \Gamma(a, b^*)$ gets tangent to the $x$-axis.  In Case 2, the curve $a \mapsto \gamma(a, b^*)$ starts at a positive value and increases monotonically (hence it is uniformly positive).  Because $b^*$ is chosen so that $\Gamma(0, b^*) = 0$, the curve $a \mapsto \Gamma(a, b^*)$ is a curve that starts at zero and monotonically increasing. 
 
 In the last row, the value functions are shown by solid lines with $a^*$ and $b^*$ indicated by down- and up-pointing triangles, respectively.  In order to confirm the optimality, we also plot $v_{a,b}$ for perturbed values of $a$ and $b$.  For Case 1 (left), we plot $v_{a,b}$ for $(a,b) = (a^*-1, b^*-1), (a^*-1, b^*+1), (a^*+1, b^*-1), (a^*+1, b^*+1)$.   For Case 2 (right), we plot $v_{a,b}$ for $(a,b) = (a^*+1, b^*-1), (a^*+1, b^*+1)$. We see that $v_{a^*, b^*}$ indeed dominates these uniformly in $x$.  In addition, the smoothness of $v_{a^*, b^*}$ as in Lemma \ref{lemma_smooth_fit} as well as the slope conditions as in Lemma \ref{ver_1} can also be confirmed.

\begin{figure}[htbp]
\begin{center}
\begin{minipage}{1.0\textwidth}
\centering
\begin{tabular}{c}
 \includegraphics[width=.5\textwidth]{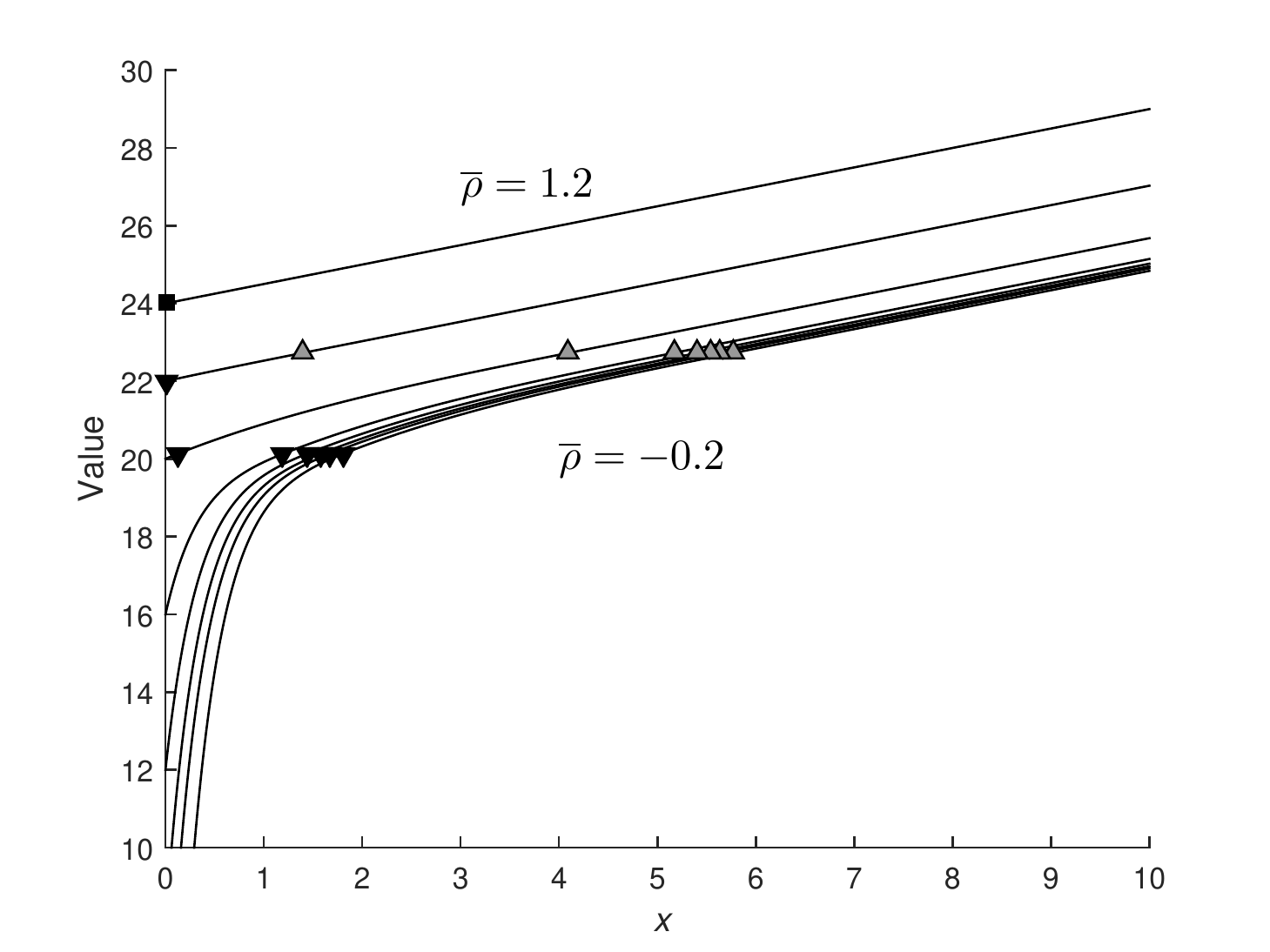} \end{tabular}
\end{minipage}
\caption{Sensitivity of the value function $v(x)$ with respect to $\bar{\rho} := q \rho / \delta$.
} \label{figure_sensitivity_rho}
\end{center}
\end{figure}

\subsection{Sensitivity analysis} \label{S_NI2}
We shall now analyze the sensitivity of $(a^*, b^*)$ and $v_{a^*,b^*}$ with respect to the parameters describing the problem. 

We first study the sensitivity with respect to $\bar{\rho} := q \rho / \delta$ with $q$ and  $\delta$ fixed.

In Figure \ref{figure_sensitivity_rho}, we plot the value function $v_{a^*, b^*}$ for $\bar{\rho}$ ranging from $-0.2$ to $1.2$. 
 The points $a^*$ and $b^*$ are, respectively, indicated by down- and up-pointing triangles; for the case $a^*=b^*=0$, these are indicated by squares.  The value function is monotonically increasing in $\bar{\rho}$ uniformly in $x$.  For sufficiently large $\bar{\rho}$ (in particular when $\bar{\rho} = 1.2$), it is optimal to liquidate ($a^*=b^*= 0$).  As we decrease the value of $\bar{\rho}$ we arrive at $0 = a^* < b^*$ (when $\bar{\rho} = 1.1$) where persistent refraction is optimal.  For sufficiently small $\bar{\rho}$, we have $0 < a^* < b^*$.  In general, we see that both $a^*$ and $b^*$ increase as $\bar{\rho}$ decreases. This is consistent with our intuition: as the terminal payoff at ruin decreases, one tries to avoid ruin and hence decreases the dividend payment.

\begin{figure}[htb]
\begin{center}
\begin{minipage}{1.0\textwidth}
\centering
\begin{tabular}{c}
 \includegraphics[width=.5\textwidth]{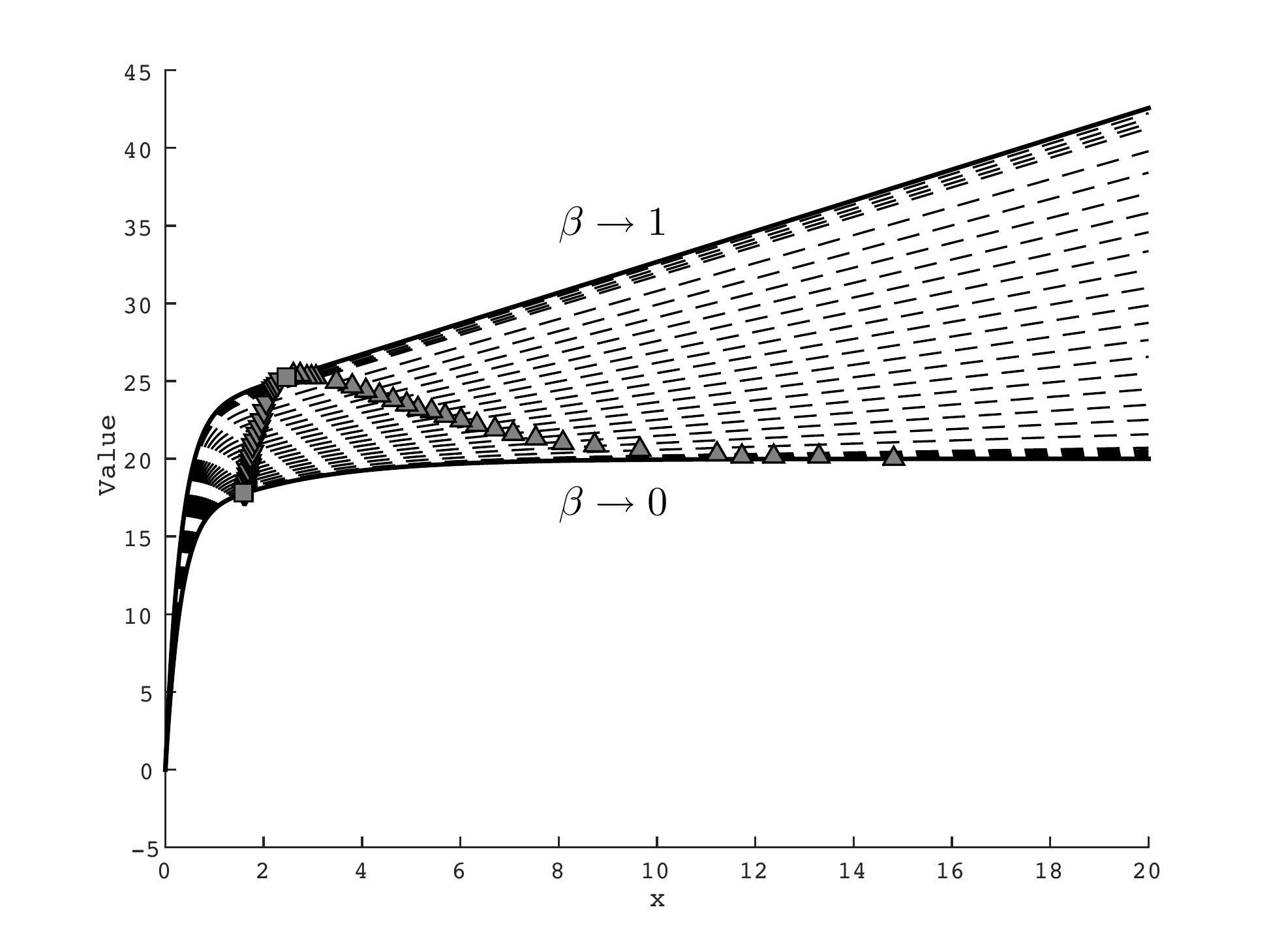} \end{tabular}
\end{minipage}
\caption{Sensitivity of the value function $v(x)$ with respect to $\beta= [0.01,0.02,0.03,0.04,0.05, 0.1, ..., 0.90,0.95, 0.96,0.97,0.98,0.99]$.
} \label{figure_sensitivity_beta}
\end{center}
\end{figure}

Next, as motivated by the results in Section \ref{subsection_convergence}, we shall analyze the relations between our problem and the following existing results:
\begin{enumerate}
\item \citet*{BaKyYa13} considered the case with only singular control that maximizes
\begin{align*}
v_{\pi}^S (x) &:= \mathbb{E}_x \left(   \int_{[0,\sigma^\pi]} e^{-q t}  \diff S_t^\pi   \right), \quad x \geq 0.
\end{align*} On condition that $\psi_Y'(0) < 0$,
 it is optimal to reflect at $b_S$ such that $\Gamma_S(b_S)=0$ where
\begin{align*}
\Gamma_S(b) :=  \overline{\mathbb{Z}}^{(q)}(b) + \frac {\psi'_Y(0+)} q, 
\end{align*}
and the value function is given by
\begin{equation} \label{value_function_bayraktar}
v^S(x) = - \overline{\mathbb{Z}}^{(q)} (b_S-x) - \frac {\psi'_Y(0+)} q, \quad x \geq 0. 
\end{equation}
\item \citet*{YiWe13} considered the case with only absolutely continuous control that maximizes
\begin{align*}
v_{\pi}^A (x) &:= \mathbb{E}_x \left(   \int_0^{\sigma^\pi} e^{-q t}  \diff A_t^\pi   \right), \quad x \geq 0.
\end{align*}
On condition that $\psi_Y'(0) < 0$,  it is optimal to refract at $a_A$ such that $\Gamma_A(a_A)=0$ where
\begin{align*}
\Gamma_A(a) := \delta \mathbb{Z}^{(q)}(a)- \frac q {\Phi(q)} e^{\Phi(q) a}\Big( 1  +\delta \Phi(q) \int_{0}^{a}\mathbb{W}^{(q)}(y) e^{-\Phi(q) y} \ud y \Big),
\end{align*}
and the value function is given by
\begin{equation} \label{value_function_yin}
v^A(x) = - e^{\Phi(q) (a_A-x)} \delta \int_0^{a_A-x} \mathbb{W}^{(q)} (z) e^{-\Phi(q) z} \diff z + \frac \delta q \mathbb{Z}^{(q)}(a_A-x) - \frac {e^{\Phi(q)(a_A-x)}} {\Phi(q)}, \quad x \geq 0. 
\end{equation}
\end{enumerate}

We now analyze the sensitivity with respect to $\beta_S = \beta$  (with $\beta_A=1$ fixed). In Figure \ref{figure_sensitivity_beta}, we plot (in dashed black)  $v_{a^*,b^*}$ for various rate $\beta$ ranging from $0.01$ to $0.99$.  The value function $v_{a^*,b^*}$ is monotonically increasing in $\beta$ uniformly in $x$.  In order to verify numerically the convergence as $\beta$ goes to $1$ and to $0$, we also plot the value functions (in solid black)  $v^S$ as in \eqref{value_function_bayraktar} and $v^A$ as in \eqref{value_function_yin} (with their optimal barriers indicated by squares). As in Section \ref{subsection_convergence} (ii), 
as $\beta$ approaches $1$ (as the cost of singular control relative to absolutely continuous control decreases to zero), the distance between $a^*$ and $b^*$ shrinks.  We confirm in the plot that the optimal barriers as well as the value function converge to those under \citet*{BaKyYa13}.   On the other hand, as the value $\beta$ decreases to zero, as in  Section \ref{subsection_convergence}  (i), $b^*$ goes to $\infty$.  We also confirm that
$a^*$ as well as the value function converge to those of \citet*{YiWe13}.

\begin{figure}[htbp]
\begin{center}
\begin{minipage}{1.0\textwidth}
\centering
\begin{tabular}{c}
 \includegraphics[width=.5\textwidth]{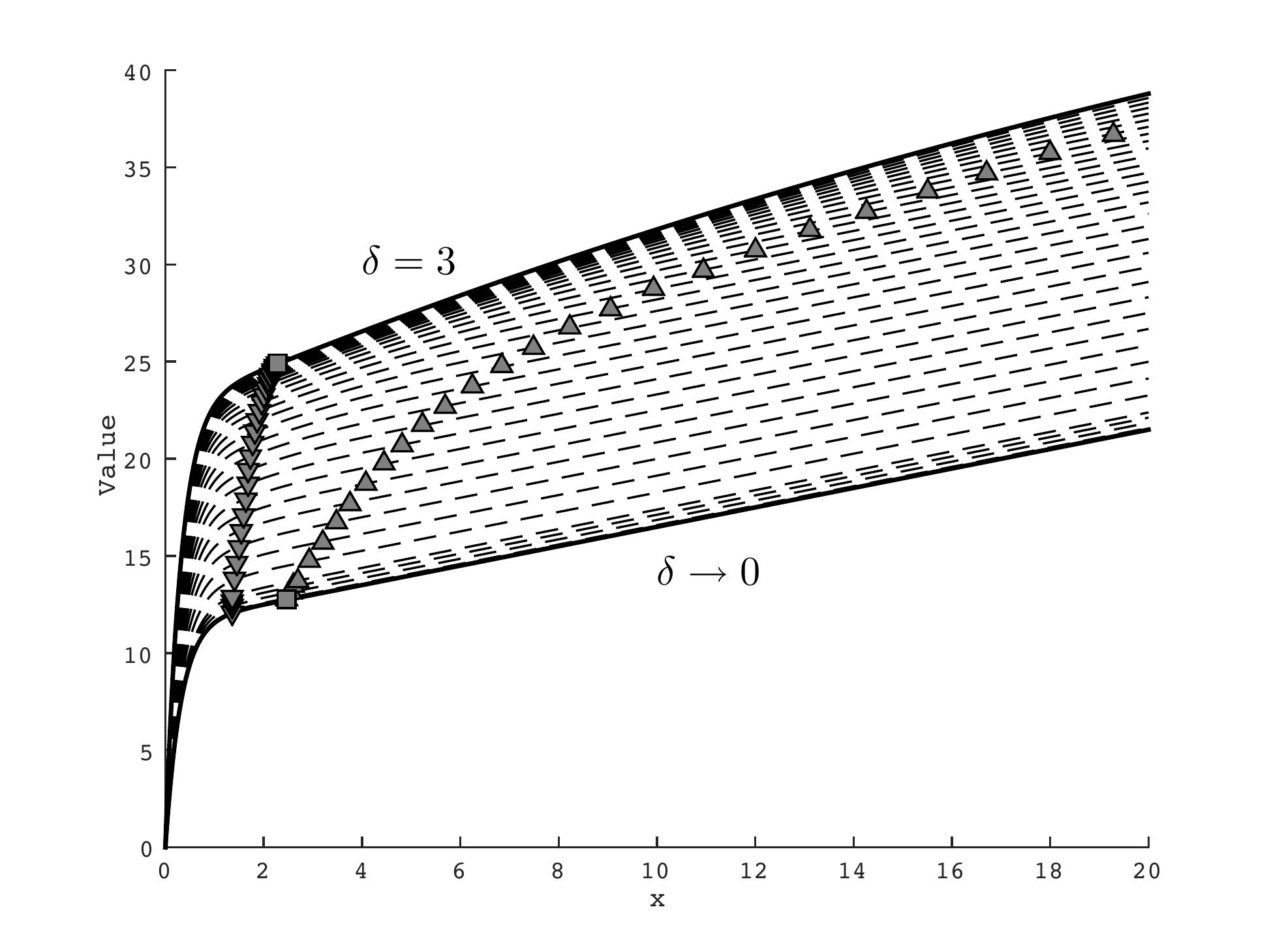} \end{tabular}
\end{minipage}
\caption{Sensitivity of the value function $v(x)$ with respect to $\delta = [0.01, 0.04, 0.07,0.1, 0.2, \ldots, 2.9,3]$.
} \label{figure_sensitivity_delta}
\end{center}
\end{figure}

Finally, we study the sensitivity with respect to $\delta$.  Figure \ref{figure_sensitivity_delta} plots (in dashed black)  the value function for $\delta$ ranging from $0.1$ to $3$ along with the value functions (in solid black) \eqref{value_function_bayraktar} and \eqref{value_function_yin} under \citet*{BaKyYa13} and \citet*{YiWe13} (with $\delta = 3$).  As $\delta$ decreases to zero, the contribution of the absolutely continuous control decreases. Consequently, $b^*$ and value function converge to those under \citet*{BaKyYa13}.   On the other hand, as we increase the value of $\delta$, as in  Section \ref{subsection_convergence} (iii) the value of $b^*$ increases quickly and hence even for a moderate value $\delta =3$, the effect of singular control gets negligible, and converges quickly to the value under \citet*{YiWe13} with the same $\delta$.

\subsection{Impact of the volatility on the optimal layer levels $a^*$ and $b^*$} \label{S_NI3}

In this section, we set $\sigma=0$,  and consider exponential jumps $Z$ with rate $\omega$. The variance (volatility) of the surplus process (per unit time) is $2\kappa/\omega^2$ and the mean of the surplus process is $\mu_Y=\kappa/\omega-c_Y$. We vary the parameter $\omega$ while keeping the ratio $\kappa/\omega$ constant as well as the parameter $c_Y$. A large value of $\omega$ corresponds to a small variance (volatility) of the process.  We plot the barrier levels $a^*$, $b^*$ and $b^*-a^*$ against the volatility of the surplus process defined in \ref{S_PTLP}.

\begin{figure}[htbp]
\centering
\includegraphics[width=0.49\textwidth]{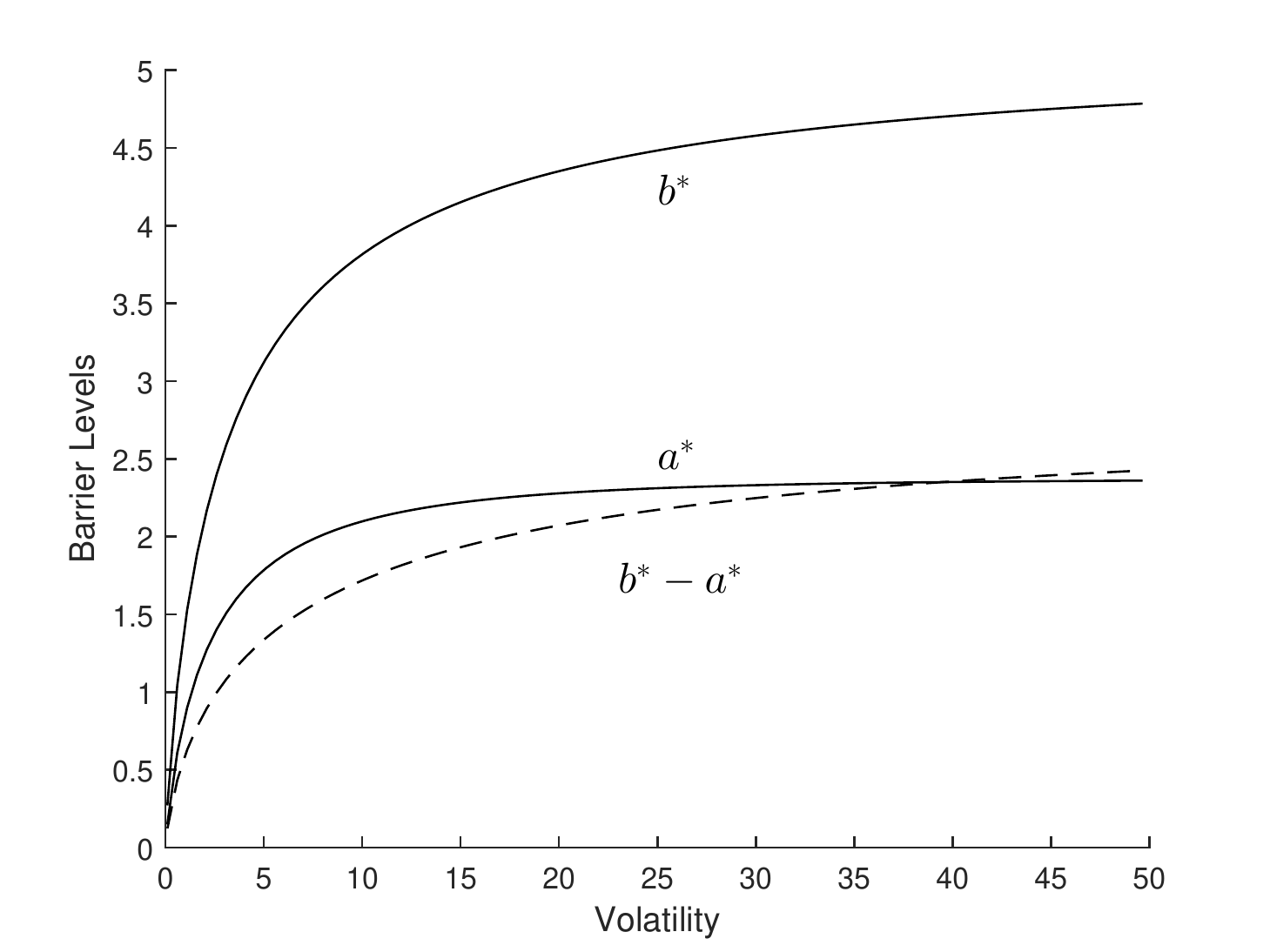}
\includegraphics[width=0.49\textwidth]{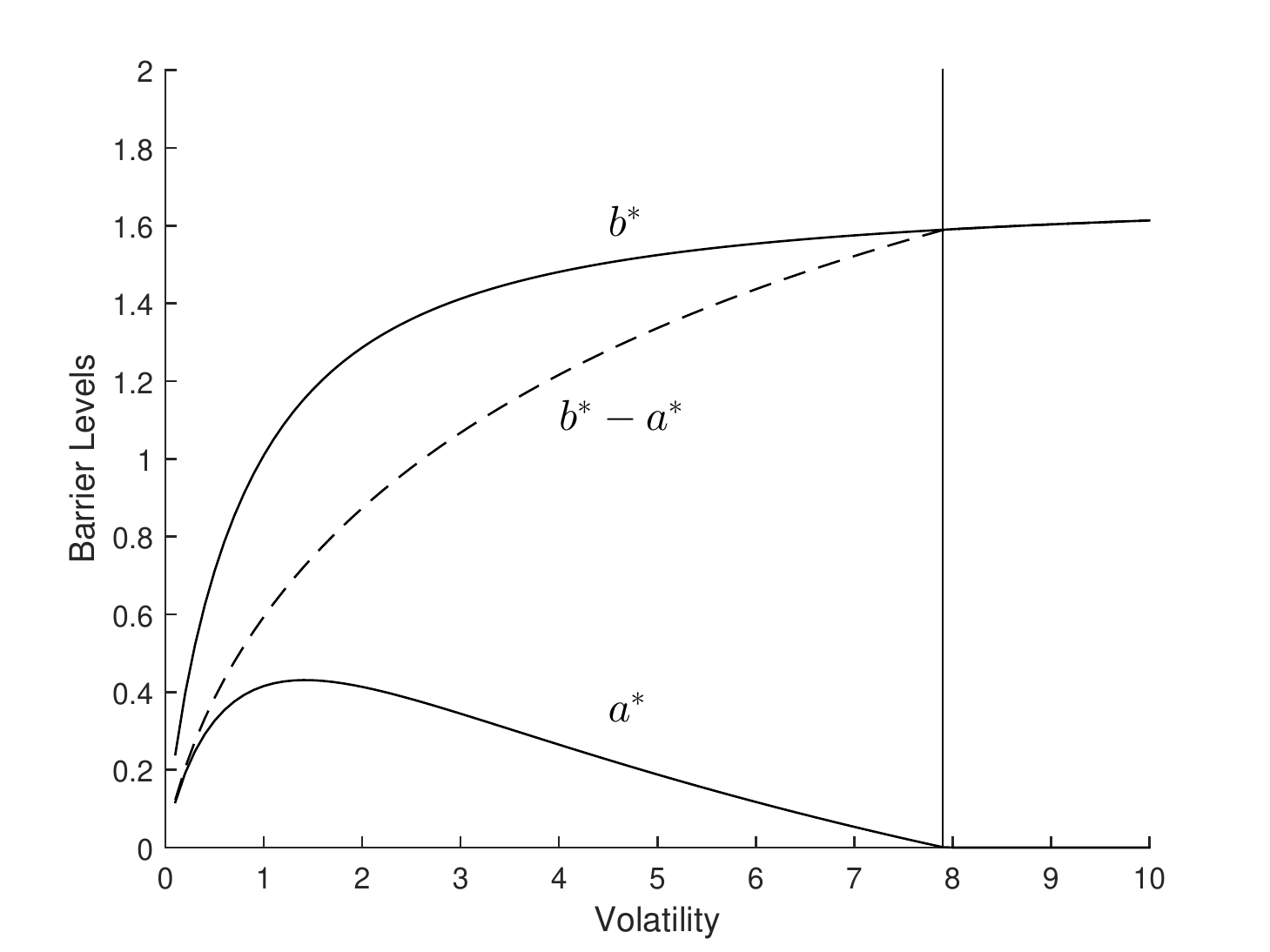}
\includegraphics[width=0.49\textwidth]{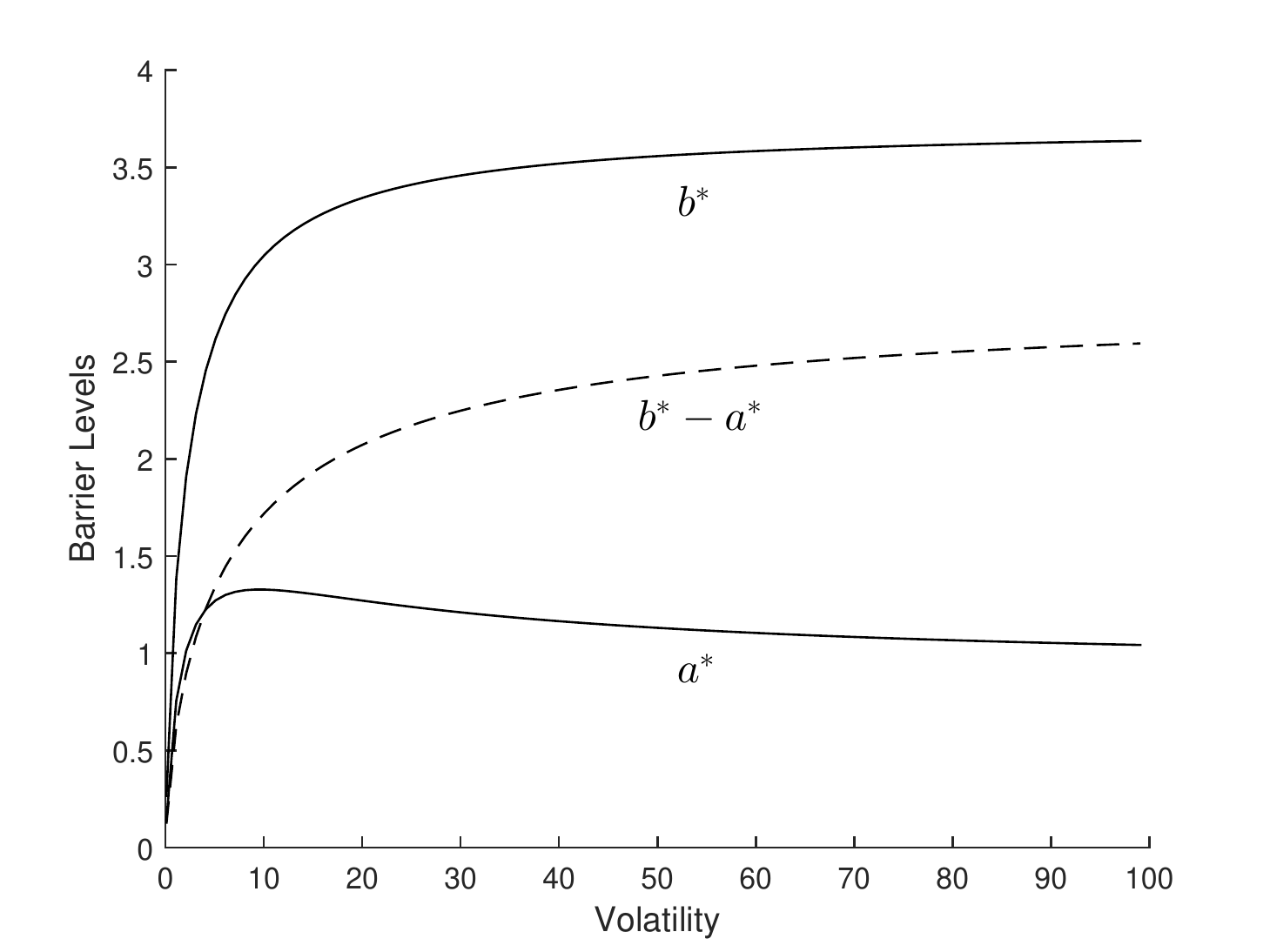}
\caption{Optimal levels $a^*$ and $b^*$ (and their difference) for varying levels of volatility, and when there is a penalty of $-\rho=2$ at ruin (top left), a reward of $\rho=2$ at ruin (top right), or absence of cash flow $\rho=0$ at ruin (bottom).}
\label{F_sensvol1}
\end{figure}

In Figure \ref{F_sensvol1}, optimal levels $a^*$ and $b^*$ (and their difference) are displayed for varying levels of volatility, and when there is a penalty of $-\rho=2$ at ruin (top left), a reward of $\rho=2$ at ruin (top right), or absence of cash flow $\rho=0$ at ruin (bottom). In all cases, both the optimal reflective barrier $b^*$, and the spread $b^*-a^*$ increase with volatility. The story about the optimal refractive level $a^*$ is less clear. In presence of a penalty, it keeps increasing. When there is no penalty or reward, $a^*$ increases initially when the volatility is increasing. However, after it reaches the peak level, it starts to decrease and shows a convergent behavior (to a non-zero value).

Interestingly, when the terminal payoff $\rho$ is large (such as in the top right case) and when volatility increases, it is initially optimal to increase the refractive level $a^*$, but then it vanishes to 0, which seems counterintuitive. In fact, this is because in this case the reward associated to ruin ($\rho>0$) becomes more interesting than future prospects when volatility is too high. However, the reflective level $b^*$ does not vanish, probably due to higher transaction costs. To some extent, this graph provides some insights into the distinction between cases (ii-1) and (ii-2) in Lemma \ref{lemma_existence}.

\section{Conclusion} \label{S_conclusion}

In this paper we considered the impact and combination of absolutely continuous (refractive) and singular (reflective) dividend payments under the assumption of a spectrally positive surplus model. By developing and proving an appropriate verification lemma, we showed that a two-layer $(a,b)$ strategy is the one that will maximise the expected present value of dividends (with any fixed liquidation payment/penalty) in this context. In fact, such a dividend strategy is observed in practice \citep[e.g.][]{GeLo12,BHP16,RioT16}. We derived the value function explicitly with the help of scale functions, and discussed the convergence of our results with the most closely related references \citet*{BaKyYa13} and \citet*{YiWe13}.

The results in this paper could possibly be generalised so as to allow the top layer to be a refractive one as well (although one should ensure that payout rates are higher in that layer). In fact, the model can be seen as \textit{multilayer} dividend strategy \citep*[see, e.g.][]{AlHa07}. We conjecture that there could be as many optimal layers with increasing maximum payout as desired, as long as the transaction costs are ranked in the same way. This is left to future research.

\section*{Acknowledgments}

The authors are grateful to a referee for helpful comments, and to Hayden Lau for research assistance. B.\ Avanzi and B.\ Wong acknowledge support under Australian Research Council's Linkage Projects funding scheme (project number LP130100723). J.\ L.\ P\'erez is supported by CONACYT, project no. 241195. K.\ Yamazaki is in part supported by MEXT KAKENHI Grant Number 26800092. The views expressed herein are those of the authors and are not necessarily those of the supporting organisations.

\appendix

\section{Specification of the numerical illustrations}

\subsection{Dual model with diffusion and phase-type jumps} \label{S_PTLP}

Let $Y$ be
of the form
\begin{equation}
 Y_t  - Y_0= - c_Y t+\sigma B_t + \sum_{n=1}^{N_t} Z_n, \quad 0\le t <\infty, \label{X_phase_type}
\end{equation}
for some $c_Y \in \R$ and $\sigma \geq 0$. 
Here $B=( B_t; t\ge 0)$ is a standard Brownian motion, $N=(N_t; t\ge 0 )$ is a Poisson process with arrival rate $\kappa$, and  $Z = ( Z_n; n = 1,2,\ldots )$ is an i.i.d.\ sequence of phase-type-distributed random variables with representation $(m, {\bm \alpha},{\bm T})$; see \citet*{AsAvPi04}.
These processes are assumed mutually independent. The Laplace exponents of $Y$ and $X$ are then (with ${\bm t} = -\bm{T 1}$ where ${\bm 1} = [1, \ldots 1]'$)
\begin{align*}
 \psi_Y(s)   &= c_Y s + \frac 1 2 \sigma^2 s^2 + \kappa \left( {\bm \alpha} (s {\bm I} - {\bm{T}})^{-1} {\bm t} -1 \right), \\
 \psi_X(s)  &= (c_Y+\delta) s + \frac 1 2 \sigma^2 s^2 + \kappa \left( {\bm \alpha} (s {\bm I} - {\bm{T}})^{-1} {\bm t} -1 \right),
 \end{align*}
which are analytic for every $s \in \mathbb{C}$ except at the eigenvalues of ${\bm T}$.  

Suppose  $( -\zeta_{i,q}; i \in \mathcal{I}_q )$ and $( -\xi_{i,q}; i \in \mathcal{I}_q )$ are the sets of the roots with negative real parts of the equalities $\psi_Y(s) = q$ and $\psi_X(s) = q$, respectively.  We assume that the phase-type distribution is minimally represented and hence $|\mathcal{I}_q| = m+1$ (if $\sigma = 0$, $|\mathcal{I}_q| = m$);
 see \citet*{AsAvPi04}.  As in \citet*{EgYa14}, if these values are assumed distinct, then
the scale functions of $-Y$ and $-X$ can be written, for all $x \geq 0$,
\begin{align}
\mathbb{W}^{(q)}(x) = \frac {e^{\varphi(q) x}} {\psi_Y'(\varphi(q))} - \sum_{i \in \mathcal{I}_q} B_{i,q} e^{-\zeta_{i,q}x} \quad \textrm{and} \quad
W^{(q)}(x) = \frac {e^{\Phi(q) x}} {\psi_X'(\Phi(q))} - \sum_{i \in \mathcal{I}_q} C_{i,q} e^{-\xi_{i,q}x},
\label{scale_function_form_phase_type}
\end{align}
respectively, where
\begin{align*}
B_{i,q} := \left. \frac { s+\zeta_{i,q}} {q-\psi_Y(s)} \right|_{s = -\zeta_{i,q}} = - \frac 1 {\psi_Y'(-\zeta_{i,q})} \quad \textrm{and} \quad
C_{i,q} := \left. \frac { s+\xi_{i,q}} {q-\psi_X(s)} \right|_{s = -\xi_{i,q}} = - \frac 1 {\psi_X'(-\xi_{i,q})}.
\end{align*}

In Sections \ref{S_NI1} and \ref{S_NI2}, we set $c_Y = 0.5$, $\sigma = 0.2$, and $\kappa = 2$ and, for the jump size distribution, we use the phase-type distribution given by $m=6$ and
 \begin{align*}
&{\bm T} = \left[ \begin{array}{rrrrrr}   -5.6546  &  0.0000   &      0.0000  &  0.0000  &  0.0000  &  0.0000 \\
    0.6066  & -5.6847 &   0.0000  &  0.0166  &  0.0089  &  5.0526 \\
    0.2156  &  4.3616  & -5.6485  &  0.9162 &   0.1424 &   0.0126 \\
    5.6247 &   0.0000   & 0.0000  & -5.6786  &  0.0000  &  0.0000 \\
    0.0107  &  0.0000  &  0.0000 &   5.7247 &  -5.7420   & 0.0000 \\
    0.0136  &  0.0000 &   0.0000  &  0.0024 &   5.7022 &  -5.7183 \end{array} \right],  \quad {\bm \alpha} =    \left[ \begin{array}{l}
    0.0000 \\
    0.0007 \\
    0.9961 \\
    0.0000 \\
    0.0001 \\
    0.0031  \end{array} \right];
\end{align*}
this gives an approximation to the Weibull random variable with parameter  $(2,1)$ \citep*[see][for the accuracy of the approximation]{EgYa14}.
For other parameters, unless stated otherwise, we set $\beta_A = 1$,  $\beta_S = 0.5$ (and hence $\beta = 0.5$),  $q = 0.05$, $\delta = 1$ (and hence $c_X = 1.5$), and $\rho  = 0$. For Section \ref{S_NI3}, the parameters used in the models are $\mu_Y=1$, $\kappa / \omega =2$, $c_Y=1$, $q = 0.2$, $\delta = 0.1$, $\beta_A = 1$,  $\beta_S = 0.6$ (and hence $\beta = 0.6$).

\section*{References}
	
\bibliographystyle{elsarticle-harv}
\bibliography{libraries}
	
\end{document}